\documentclass[12pt]{article}
\usepackage{amsmath,amssymb,amsfonts,amsthm,url,xspace,graphicx,color,mathrsfs,enumitem}
\usepackage{tikz}
\usepackage{bm}
\usepackage[colorlinks=true,linkcolor=blue,citecolor=blue,urlcolor=blue,pdfborder={0 0 0}]{hyperref}
\usepackage{graphicx}
\usepackage{subcaption}
\usepackage{float}

\usepackage{xcolor}

\theoremstyle{plain}
\newtheorem{thm}{Theorem}[section]

\newtheorem{lemma}[thm]{Lemma}

\newtheorem{cor}[thm]{Corollary}

\newcommand{\eps}{\varepsilon}
\newcommand{\C}{\mathbb{C}}
\newcommand{\Z}{\mathbb{Z}}

\newcommand{\G}{\mathcal{G}}

\newcommand{\K}{\mathbb{K}}

\newcommand{\E}{\mathbb{E}}
\newcommand{\R}{\mathbb{R}}
\newcommand{\T}{\mathbb{T}}
\newcommand{\N}{\mathbf{N}}
\newcommand{\NN}{\mathbb{N}}
\newcommand{\M}{\mathcal{M}}
\newcommand{\x}{\mathrm{x}}
\newcommand{\X}{\mathbf{x}}

\newcommand{\Cov}{\mathrm{Cov}}
\newcommand{\Var}{\mathbf{Var}}

\renewcommand{\Re}{\operatorname{Re}}
\renewcommand{\Im}{\operatorname{Im}}
\renewcommand{\P}{\mathbb{P}}
\renewcommand{\t}{\mathrm{t}}
\renewcommand{\b}{\mathrm{b}}
\newcommand{\w}{\mathrm{w}}

\newcommand{\be}{\begin{equation}}
\newcommand{\ee}{\end{equation}}

\newcommand{\old}[1]{}

\title{The multinomial tiling model}
\author{Richard Kenyon\footnote{Department of Mathematics, Yale University, New Haven; \{richard.kenyon,andrei.pohoata\}@yale.edu.} \and Cosmin Pohoata\footnotemark[1]}
\begin{document}
\maketitle
\abstract{
Given a graph $\G$ and collection of subgraphs $T$ (called tiles), 
we consider covering $\G$ with copies of tiles in $T$
so that each vertex $v\in\G$ is covered with a predetermined multiplicity. The \emph{multinomial tiling model} 
is a natural probability measure on such configurations (it is the uniform measure on standard tilings of the corresponding  ``blow-up'' of $\G$).

In the limit of large multiplicities we compute asymptotic growth rate of the number of multinomial tilings. We show that the individual tile densities tend to a Gaussian field with respect to an associated discrete Laplacian.
We also find an exact discrete Coulomb gas limit when we vary the multiplicities.

For tilings of $\Z^d$ with translates of a single tile and a small density of defects, we study a crystallization phenomena when the defect density tends to zero, and give examples
of naturally occurring quasicrystals in this framework.
}

\section{Introduction}

The study of random tilings is a cornerstone area of combinatorics and statistical mechanics.
In its simplest form, the random tiling model consists in the study of the set of 
tilings of a region (for example a subset of the plane) with translated copies of a finite collection
of shapes, called prototiles. However even the simplest cases can lead to 
hard problems. The mere existence of a tiling of a region in $\R^2$ 
with a prescribed set of polyominos is an NP-complete problem \cite{Levin}, even if the prototiles
consist in just the $3\times1$ and $1\times3$ rectangles \cite{BNRR}. 
Enumerating tilings is of course even harder.

However in the few cases where we \emph{can} analyze random tilings, like random ``domino" tilings
(tilings with $2\times 1$ and $1\times 2$ rectangles) or ``lozenge" tilings (tilings with $60^\circ$ rhombi),
we find very rich behavior, with beautiful enumerative properties \cite{Kast, TF, EKLP, EKLP2, Macmahon}, phase transitions \cite{KOS}, limit shapes \cite{KO}, conformal invariance \cite{K.ci}, and Gaussian scaling limits \cite{K.GFF}. 
Beyond these and other dimer models there are almost no other cases we can analyze in detail. 
There are other cases where enumeration is sometimes possible, like the $6$-vertex model \cite{Lieb}, but for these models very little is known about 
correlations, although they are sometimes predicted in physics to be Gaussian in the scaling limit and ``conformally invariant''---such models were in fact the inspiration for conformal field theory.

We study here a variant of the random tiling problem: the \emph{multinomial tiling problem}, 
which is tractable in the sense that we can give 
exact generating functions for enumerations, which in turn yield, in the limit of large multiplicity,
exact asymptotic expressions
for growth rates and Gaussian behavior for random tilings. This setting is quite general
and works for tilings in arbitrary graphs, not just plane regions. 
Furthermore we 
find all of the phenomena discussed above: phase transitions, limit shapes, 
crystallization phenomena, and conformal invariance (which we study in \cite{KP}). 

It comes as an additional surprise that in certain situations our random tilings form \emph{quasicrystals}.
Quasicrystals were first found in nature by Schechtman et al \cite{Schechtman}. Their physical and mathematical framework is still debated, but examples
of quasiperiodic tilings were first found by Berger \cite{Berger} and familiar examples like Penrose tilings are
now well understood \cite{deBruijn}; they are sets of tiles which tile the plane but only in nonperiodic fashion. 
Our quasicrystals arise from \emph{random} tilings;
it is the statistical correlations between tile densities that are quasiperiodic (even though there are periodic points in the configuration space). For other examples of random 
quasicrystals, see for example \cite{Kalugin, dGN}.

Let $\G=(V,E)$ be a finite graph and 
let $T=\{t_1,\dots,t_k\}$ be a collection of subgraphs of $\G$, called \emph{tiles}.
Let $\N=\{N_v\}_{v\in\G}$ be nonnegative integers associated to vertices of $\G$. 
Define a new graph $\G_{\N}$, the ``$\N$-fold blowup'' of $\G$, to be the graph 
obtained by replacing each vertex $v$ of $\G$ with $N_v$
vertices, and each edge $uv$ with the complete bipartite graph $K_{N_u,N_v}$.
Now each tile $t\in T$ can be lifted to a subset of $\G_{\N}$ in many ways:
if $t$ has vertices $v_1,\dots,v_k$ then it has $N_{v_1}\cdots N_{v_k}$-many lifts.

We consider \emph{tilings} of $\G_{\N}$ with lifts of copies of tiles in $T$ (a tiling is a partition of the vertices of
$\G_{\N}$ into disjoint sets each of which is a lift of a single tile of $T$). 
Let $\Omega(\N)$
be the set of all tilings; we call these \emph{$\N$-fold tilings}.

Let $w:T\to\R_{>0}$ be a positive real weight assigned to each tile.
An $\N$-fold tiling $m$ is assigned a weight $w(m) := \prod_{t\in T}w_t^{m(t)}$
where $m(t)$ is the number of copies of $t$ used. 
The \emph{partition function} for $\N$-fold tilings is defined to be
$$Z(w,\N) = \sum_{m\in \Omega(\N)} w(m).$$
We let $\mu=\mu(\N)$ be the natural probability measure on $\Omega(\N)$, giving an $\N$-fold tiling
a probability proportional to its weight. 

We note that $\mu$ is \emph{not} the same as the uniform measure on tilings of $\G$ covering each vertex $N_v$ times;
each such ``multiple tiling'' of $\G$ can be typically lifted to a tiling of $\G_{\N}$ in many ways.

\subsection{Results}

We compute a generating function for $Z(w,\N)$ (Theorem \ref{expgf}), and 
the asymptotic growth rate of $Z(w,\N)$ as $\N\to\infty$, see (\ref{sigma}). 
This computation involves solving a nonlinear system of equations (\ref{criticalwts});
however the solution is realized as the unique minimizer of a convex functional 
(Theorem \ref{divgauge}).

In Theorem \ref{cov} we show that in the $\N\to\infty$ limit the tile occupation fractions tend to a
\emph{Gaussian field}
governed by a discrete Laplacian operator $\Delta$ on $\G$, the \emph{tiling laplacian}. 

For transitive graphs, when we vary the multiplicities, we obtain a \emph{Coulomb gas}: defects in multiplicity interact via Coulombic 
potentials arising from $\Delta$ (see Section \ref{coulomb}).

Under certain conditions on transitive graphs, our random multinomial tilings also undergo a crystallization phenomenon,
where the correlations between distant tiles no longer decay; the tiling freezes into a periodic or quasiperiodic state. This occurs on $\Z^2$, for example, tiled with translates of the $L$-triomino and a small density of singleton
monomers (Section \ref{L}). As the density of monomers tends to zero the correlation length 
of the system tends to infinity, and the system \emph{freezes}. 
There is a spontaneous symmetry breaking, since there are three distinct
crystalline states (corresponding to the three distinct---up to translation---periodic
tilings of the plane with $L$ triominos). 

For certain other polyominos we get similar freezing
phenomena, and others we don't; the behavior 
depends on the presence and type of
zeros of the underlying \emph{characteristic polynomial} $p(z,u)$
on the unit torus $\T^2\subset\C^2$. 
If the (isolated) zeros on $\T^2$ are sufficiently ``generic", we show that the resulting tiling will be a quasicrystal
(Section \ref{qp2}).
However there is a plethora of nongeneric behavior for the roots of $p$ as the tile type varies, 
yielding a similarly wide variety
of behaviors for random tilings (Section \ref{other}).

\bigskip
\noindent{\bf Acknowledgments:} We thank Wilhelm Schlag and Jim Propp for helpful conversations. R.K. was supported by NSF DMS-1854272, DMS-1939926 and the Simons Foundation grant 327929.

\section{Combinatorics}

It is convenient to generalize our definition of tile, to allow the vertices of a tile to have multiplicity larger than 
one. The vertices of a tile $t$ then form a \emph{multiset} of vertices of $\G$, that is, a subset in which each vertex $v$ has a nonnegative integer multiplicity $t_v$. We identify a tile with its multiset. If we want to think about a tile as
a subgraph, we take all edges of $\G$ connecting vertices which have positive multiplicity in $t$.
A lift of a tile $t$ to $\G_\N$ corresponds to a choice, for each $v\in\G$, of $t_v$ distinct vertices of $\G_\N$
lying above $v$, along with the set of all edges of $\G_\N$ joining these vertices (that is, the induced subgraph
of $\G_\N$ on these vertices). Each lift of $t$ is a blow up
of (the subgraph underlying) $t$.

\subsection{Generating function}

We associate a variable $x_v$ to each vertex $v\in\G$. To each tile $t\in T$ is associated the monomial $x_t = \prod_{v\in\G} \frac{x_v^{t_v}}{t_v!}$.
(Here the factor $t_v!$ accounts for the indistinguishability of the vertices of the same type in a lift of $t$.)
Let $P=P(x_1,\dots,x_{V})$ be the polynomial $P=\sum_{t\in T}w_tx_t$. We call $P$ the \emph{tiling polynomial}.
The function $F(X_1,\dots,X_V) := \log P(e^{X_1},\dots,e^{X_V})$ is called the \emph{free energy}  (see 
Section \ref{freeenergysection} below).

\begin{thm}\label{expgf}
Let $\x^\N = \prod_{v}x_v^{N_v}$. Then
$$Z(w) := \exp(P) = \sum_{\N\ge0} Z(w,\N)\frac{\x^\N}{\N!}$$
where the sum is over all vectors of nonnegative multiplicities. If we fix the total number $K$ of tiles then the corresponding 
generating function is $P^K/K!$. 
\end{thm}

\begin{proof} 
Suppose we use tile $t$ with multiplicity $K_t$. Label the abstract copies of tile $t$ with labels $\ell\in\{1,\dots,K_t\}$. To place those tiles in $\G_{\N}$,
at each vertex $v$ of $t$, we must choose $K_t$ subsets of size $t_v$  (one of each label $\ell$) out of the $N_v$ vertices of $\G_\N$ lying over $v$.  Taking into account all tiles, this is a multinomial coefficient
at vertex $v$: it is $\frac{N_v!}{\prod_t(t_v!)^{K_t}}$ total choices. We take the product of these over all vertices, and then need to divide
by $\prod_t K_t!$, the set of choices of initial labellings. In total,
the number of tilings with tile multiplicities $K_t$ and vertex multiplicities $\N$ is
$$\frac{\prod_vN_v!}{\prod_t K_t!\prod_{v,t}(t_v!)^{K_t}}.$$ Multiplying by $\prod_t (w_t\prod x_v^{t_v})^{K_t},$ 
the tile weights (and factors of $x$), dividing by $\N!$, this is
$$\prod_t \frac{(w_t\prod_v\frac{x_v^{t_v}}{t_v!})^{K_t}}{K_t!}$$
and summing over the $K_t$s gives the result.
\end{proof}

\subsection{Feasible multiplicities}

For a given graph $\G=(V,E)$ and tiling set $T$, not all multiplicities $\N\in (\Z_{\ge0})^{V}$ are feasible. 
The set of feasible multiplicities $\M_\Z=\M_\Z(T,\G)$ is (just by definition) the set of 
nonnegative integer linear combinations of vectors $v_t= \sum_{v\in V} t_ve_v$,
where $\{e_v\}_{v\in V}$ are the standard basis vectors for $\Z^V$. 

In other words we have $\M_{\Z} = D((\Z_{\ge0})^T)$ where $D:\R^T\to \R^V$ is the linear map defined by 
$D(e_t) = \sum_{v\in V}t_ve_v$.
In the standard basis the matrix of $D$ (which we also denote $D$) is called 
the \emph{incidence matrix} of the tiling problem: $D=(D_{v,t})$ where $D_{v,t} = t_v$, the multiplicity of $v$ in $t$. 
Feasible multiplicities $\M_{\Z}$ are certain integer points in a real polytopal cone $\M_\R\subset\R^V$; 
$\M_\R=D((\R_+)^{T})$.  

Typically not all integer points in $\M_\R$ are in $\M_\Z$.
For example if all tiles have size $\delta$ then necessarily $\N$ sums
to a multiple of $\delta$. More generally if $\phi$ is a homomorphism
from $\Z^V$ to some abelian group, with the property that $\phi(D(e_t))=0$ for all tiles $t$
then $\phi(\N)=0$ as well. In the language of tilings these are
called ``coloring" conditions. 

As a typical example of a coloring condition,
suppose we wish to tile $\Z^2$ or a subgraph of it with translates of
bars of length $3$: translates of $\{(0,0),(1,0),(2,0)\}$
and $\{(0,0),(0,1),(0,2)\}$. Let $\phi:\Z^{\Z^2}\to \Z/3\Z$ be defined
by $\phi(e_{(x,y)})=x\bmod 3$. 
Note that
$\phi$ applied to the translate of any tile is zero:
$$\phi(e_{(x,y)}+\phi(e_{(x+1,y)})+\phi(e_{(x+2,y)}) = 0 = \phi(e_{(x,y)})+\phi(e_{(x,y+1)})+\phi(e_{(x,y+2)}).$$
We conclude that
$\phi(\N)=0$ for any feasible multiplicity. 
This is equivalent to saying that $\N$
must include an equal number of vertices in each of the three translates of the sublattice $\{(x,y)\in\Z^2~|~x+y\equiv 0\bmod 3\}$. The same argument  with 
$\phi_-(e_{(x,y)}) = y\bmod 3$ gives another linear constraint on $\N$. 

\subsection{Homology}

The incidence map $D:\R^T\to\R^V$ is generally neither surjective nor injective. Letting $D^*$ be its transpose with respect
to the standard bases, we write 
$\R^V\cong \Im(D)\oplus \ker(D^*)$ and $\R^T\cong \Im(D^*)\oplus \ker(D)$.
These are orthogonal decompositions with respect to the standard inner products. The map 
$D$ is an isomorphism from $\Im(D^*)$ to $\Im(D)$, and likewise 
$D^*$ is an isomorphism from $\Im(D)$ to $\Im(D^*)$
 
We define $H_1(T,\R):=\R^V/\Im(D)\cong \ker(D^*)$. 
Over the integers we define $H_1(T,\Z):=\Z^V/\Im(D)$ 
to be the cokernel of the map $D$. Colorings 
$\phi$ are then elements of $H_1$,
that is, are functions on vertices which sum to $0$ for each tile: $D^*\phi(t)=\sum_v t_v\phi(v) = 0$.

If all tiles have the same size $\delta$, then $H_1(T,\Z)$ contains a copy of $\Z/\delta\Z$; the corresponding coloring functions are constant functions $f:V\to\Z/\delta\Z$.

For the above example with bars of length $3$, consider tilings of an $n\times n$ grid,
$n\ge 3$.
Then $H_1(T,\R)\equiv \R^4$: an element of $H_1(T,\R)$ is determined by its values
on the lower left $2\times 2$ square in the grid, which can be arbitrary reals. 

The existence of nontrivial integer constraints has an effect on the long-range
behavior of random tilings, see Section \ref{crystalsection} below.

\subsection{Laplacian}

The \emph{tiling laplacian} $\Delta:\R^V\to\R^V$ is the operator $\Delta=DCD^*$,
where $C$ is the diagonal matrix of tile weights $w_t$.
It has matrix $\Delta=(\Delta_{u,v})_{u,v\in V}$ with 
\be\label{lapdef}\Delta_{u,v} = \sum_{t \in T}w_tt_ut_v.\ee
Equivalently, for $f:V\to\R$ we have 
$$(\Delta f)(v) = \sum_u (\sum_tw_tt_ut_v)f(u).$$ 
The laplacian controls the covariances between tile densities, see Section \ref{covsection} below.

\subsection{Gauge equivalence}\label{gaugesection}

Tile weight functions $w,w'$ on $T$ are said to be \emph{gauge equivalent} if 
there is a positive function $f:V\to\R_+$ such that for all $t\in T$,
$w'_{t} = w_{t}\prod_{u\in t}f(u)$. 
We call $f$ a \emph{gauge transformation}.

\begin{lemma} For fixed multiplicities $\N$, 
gauge equivalent weight functions give the same probability measure on multinomial tilings.
\end{lemma}

\begin{proof} Suppose $w'$ is gauge equivalent to $w$, that is 
$w'_{t} = w_{t}\prod_{v}f(v)^{t_v}$. An $\N$-fold tiling $m$ for weights $w'$ has weight 
\begin{eqnarray*}
\prod_t (w'_t)^{m(t)}&=&\prod_t \left(w_t^{m(t)}\prod_{v}f(v)^{t_vm(t)}\right) \\
&=&  \left(\prod_t w_t^{m(t)}\right)\prod_vf(v)^{\sum_{t}t_vm(t)} \\
&=&  \left(\prod_t w_t^{m(t)} \right)\prod_{v}f(v)^{N_v}.
\end{eqnarray*}
In particular its weight for $w'$ is equal to its weight for $w$ multiplied by a constant independent of $m$.
\end{proof}

Note that if $w'$ is gauge equivalent to $w$, then the gauge transformation $f:\G\to\R_+$ from $w$ to $w'$
may not be unique: the set of functions $f$ satisfying $\prod_{v\in t}f(v) = 1$ for all $t$ is by definition
the kernel of $D^*$, written multiplicatively (that is, $\log f\in H_1(T,\R)$).

\section{Asymptotics}\label{asympsection}

In this section compute the asymptotic growth of $Z(w,\N)$ as $\N\to\infty$.

\subsection{Fixing the number of tiles}
\label{number}

Given the multiplicities $\N$, is convenient to also fix the total number of tiles $K$.
If all tiles have the same size $\delta$, then the total number of tiles $K$
is determined by the multiplicities $\N$: we have $K=nN/\delta$ where $n=|V|$. 
More generally we proceed as follows.

We adjoin a new ``dummy" vertex $v_0$ to $\G$, connected to all other vertices. Let $\tilde\G=\G\cup\{v_0\}$ be this new graph.
We add to each tile a number of copies of the dummy vertex $v_0$ 
so that all tiles now have the same size $\delta$.
Let $x_0$ be a variable associated to the new vertex $v_0$,
and let $P_0$ be the new tiling polynomial; it is a homogenization of $P$,
replacing a monomial $z$ by $\frac{x_0^m}{m!}z$, where $m+\deg(z)=\delta$.
The number $\delta$ is the degree of $P_0$, and the size of every tile.

Let $N_0$ be an arbitrary multiplicity at $v_0$, and 
$M=\sum_vN_v$ be the total multiplicity of the other vertices (not including $v_0$).
For tileability we need $M+N_0$ to be a multiple of $\delta$: 
\be\label{MNK}M+N_0=K\delta.\ee
Note then that given the remaining multiplicities, the choice of $N_0$ is linearly related to the number of tiles $K$.

We assume for the rest of the paper, unless explicitly stated, that {\bf all tiles have the same size $\delta$}.
Notationally we can then use $\G$ instead of $\tilde\G$ and $P$ instead of $P_0$.

\subsection{Saddle point}

For each $v\in V$ let $\alpha_v\in\R_+$ be fixed. Let $\vec\alpha=(\alpha_v)_{v\in V}$. We suppose 
$\vec\alpha\in\M_{\R}(\G)$, that is, $\vec\alpha$ is in the cone of feasible multiplicities.
Take $N_v\to\infty$ simultaneously for each $v$, in such a way that each $\N\in\M_\Z(\G)$,
and $\frac{N_v}{K}\to\alpha_v$.
The quantity $\alpha_v$ is the (asymptotic) fraction of tiles covering $v$, and
\be\label{alphasum}\sum_v\alpha_v=\delta.\ee

From Theorem \ref{expgf} we have
$$\frac{K!}{\N!}Z(w,\N) = [\x^\N]P^K.$$
We extract the coefficient of $\x^\N$ of $P^K$ using a contour integral:
$$\frac{K!}{\N!}Z(w,\N) = \frac1{(2\pi i)^V}\int_{(S^1)^V} \frac{P^K}{\prod_{v} x_v^{N_v}}\prod_{v}\frac{dx_v}{x_v}.$$

For large $K$ we use the saddle-point method.
The saddle point is located at the critical point of the integrand, which is defined by the equations, 
one for each $v\in \G$:
$$\frac{\partial}{\partial x_v}(K\log P-\sum_v N_v\log x_v) = 0$$
or in the large-$K$ limit
\be\label{criticalwts}\frac{x_v(P)_{x_v}}{P}=\alpha_v.\ee

Solutions to (\ref{criticalwts}) are discussed in Theorem \ref{divgauge} below. 
Although positive solutions always exist, they are not in general unique; however two positive solutions 
differ only by a gauge equivalence in $H_1(T,\R)$, and as a consequence give rise to the same weight function and growth rate. (See Section \ref{ex2} below for an example with nonuniqueness.)

For a solution $\x = \{x_v\}_{v\in V}$ the growth rate  of $\frac{K!}{\N!}Z(w,\N)$ is 
\begin{align}\nonumber \sigma(w,\{\alpha_v\}) &:= \lim_{K\to\infty}\frac1{K}\log \frac{K!}{\N!}Z(w,\N)\\
&\label{sigma}= \log P(\X)-\sum_v\alpha_v\log \x_{v}.
\end{align}
We call $\sigma(w,\{\alpha_v\})$ the \emph{exponential growth rate} of the multinomial tiling model.

Scaling so that $P=1$, the criticality equations (\ref{criticalwts}) can be written: for all $v\in V$, 
\be\label{simplecriteqns}\sum_t w_tt_vx_t = \alpha_v.\ee

\subsection{Critical gauge}\label{freeenergysection}

\begin{thm}\label{divgauge} For any $\vec\alpha\in \M_\R(\G)$ and weight function $w$ there is a unique gauge equivalent weight function $w'$ with the property
that for all $v$ the sum of weights of tiles containing vertex $v$ (counted with multiplicity) is $\alpha_v$, that is $\sum_v w'_tt_v = \alpha_v$.  A corresponding
gauge transformation $f:V\to\R_{>0}$ solves the criticality equations (\ref{criticalwts}) with $x_v=f(v)$.
\end{thm}

We call $w'$ of this theorem the (weight function in the) \emph{critical gauge}. 

\begin{proof} 
Define $X_v=\log x_v$. 
The free energy $F(X_1,\dots,X_V) := \log P(e^{X_1},\dots,e^{X_V})$ is a smooth function of the $X_i$'s. 
It has gradient lying in $\Im(D)$: its gradient is 
$$\nabla F = (\frac{x_1P_{x_1}}{P},\dots,\frac{x_TP_{x_T}}{P}) = \frac1PD(\sum_tw'_te_t),$$
where $w'_t = w_t\prod_{v\in t}x_v$.

Moreover we claim that $F$ is convex.
If we interpret $P$  (after scaling so that $P(1)=1$) as the probability generating function for a $\R^V$-valued random variable $Y$,
then the Hessian matrix $H$ of $F$ is $H_F=(\frac{\partial^2 \log P}{\partial X_u\partial X_v})_{u,v\in V}$
is the covariance matrix of $Y$, hence positive semidefinite. 
$F$ is strictly convex on directions in $\Im(D)$, as these are directions where the variance is positive,
 and $F$ is constant on directions
in $\ker(D^*)$, that is, those perpendicular to $\Im(D)$.

Let $S(\vec\alpha)$ be the Legendre dual of $F$: for $\vec\alpha\in\M_\R\subset\Im(D)$, we have
$$S(\alpha_1,\dots,\alpha_V)= \max_{X_1,\dots,X_V} \left\{-\log P(e^{X_1},\dots,e^{X_V}) +\alpha_1X_1+\dots+\alpha_VX_V\right\}.$$
Then $S$ is strictly convex and defined on all of $\M_\R\cap\{\sum_v\alpha_v=\delta\}$.  We have 
$$\alpha_v = \frac{\partial}{\partial X_v}\log P(e^{X_1},\dots,e^{X_V})$$
so that the criticality equations are satisfied with $\x_v = e^{X_v}$. Comparing with (\ref{sigma}) we see that 
$-S(\vec\alpha) = \sigma(w,\vec\alpha)$ is the growth rate function. 

Here the maximizing $X_v$
are unique up to a global additive constant and up to translations in $\ker D^*$; the latter 
correspond precisely to gauge transformations not
changing the tile weights. The former allow us to scale all weights so that $P(1)=1$. 
After this scaling (\ref{criticalwts}) or (\ref{simplecriteqns}) say precisely that the sum of weights of tiles containing $v$ (counted with multiplicity)
is $\alpha_v$.
\end{proof}

Since $\sigma$ is strictly concave, a solution to (\ref{criticalwts}) can be found by maximizing $\sigma$,
written as a function of the $x$'s (even though $\sigma$ is not strictly concave when written as a function of the $x$'s
since it is constant on directions in $H_1$, it is strictly concave on orthogonal directions.)

\begin{cor}\label{probs} For the critical gauge $w'$, tile probabilities are proportional to tile weights, that is, the expected number of 
tiles of type $t$ is $Kw'_t=Kw_t\x_t$, where $K$ is the total number of tiles.
\end{cor}

One consequence of this corollary is that there is, for any choice of tile probabilities (satisfying the 
necessary condition of summing to $\alpha_v$ at vertex $v$ for each $v$), a choice of tile weights $w_t$,
unique up to gauge, for which the multinomial tiling model has those tile probabilities.

\subsection{Example}

Consider tilings of $\G_1 = \{1,2,3,4,5\}\subset\Z$ with tiles consisting of single vertices and
pairs of adjacent vertices:
the tiles are $T=\{1,2,3,4,5,12,23,34,45\}.$
We add dummy vertex $v_0$ and let $\G=\G_1\cup\{v_0\}$ where $0$ is connected to all vertices of $\G_1$. 
Suppose $N_i=N$ for $i\ne v_0$, and all tile weights are $1$. Then
$$P = x_0(x_1+x_2+x_3+x_4+x_5) + x_1x_2+x_2x_3+x_3x_4+x_4x_5.$$
We have $N_0+5N=2K$. Let $\alpha = N/K$ and $\alpha_0=N_0/K$ (note $\alpha_0+5\alpha=2=\delta$). 

The feasible range of $\alpha$ is $\alpha\in[\frac15,\frac13]$:
when $\alpha=1/5$, $K=5N$ and we need to use only singleton tiles,
and when $\alpha=1/3$, $K=3N$ and we need to use the maximum proportion of long tiles (which is two long tiles for every singleton tile); moreover the 
singleton tiles must be
$x_1,x_3$ or $x_5$.

Solving
the criticality equations (\ref{criticalwts}) we find the tile probabilities
\begin{align*}x_0x_1&=\frac14(1+3\alpha-\sqrt{1-10\alpha+41\alpha^2})\\
x_0x_2&=\frac12(1-3\alpha)\\
x_0x_3&=\frac12(1-7\alpha+\sqrt{1-10\alpha+41\alpha^2})\\
x_1x_2&=\frac14(-1+\alpha+\sqrt{1-10\alpha+41\alpha^2})\\
x_2x_3&=\frac14(-1+9\alpha-\sqrt{1-10\alpha+41\alpha^2})\\
\end{align*}
and the remaining probabilities are given by symmetry. 

Tile probabilities are plotted in Figure \ref{tileprobs5}.
\begin{figure}
\begin{center}\includegraphics[width=2.5in]{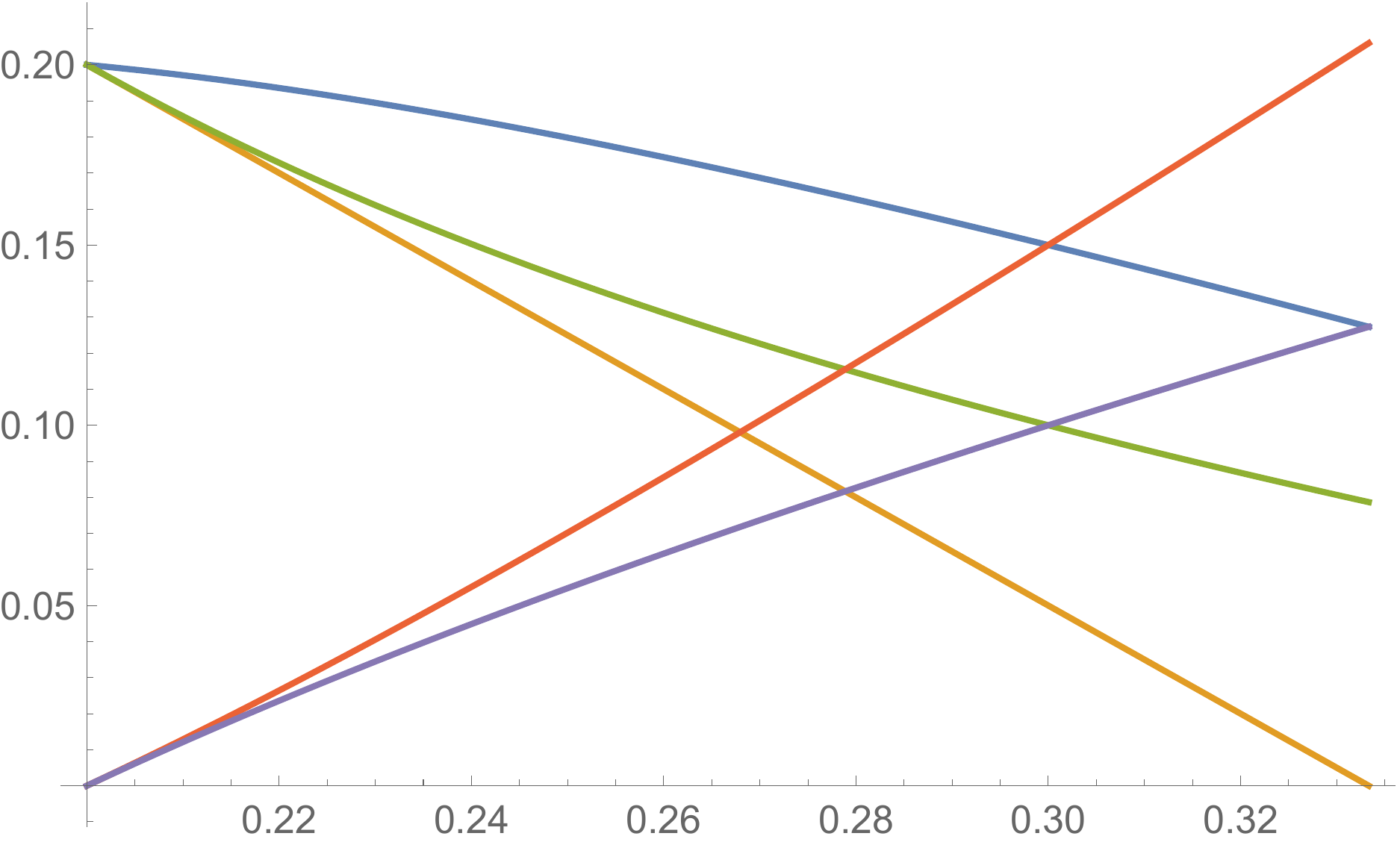}\end{center}
\caption{\small{\label{tileprobs5}Tile probabilities for $\alpha\in[1/5,1/3]$. $x_0x_1$ in blue, $x_0x_2$ in orange, $x_0x_3$ green, $x_1x_2$ red, $x_2x_3$ purple. }}
\end{figure}

\subsection{Example}\label{ex2}

Here is an example with nontrivial homology.
Consider tilings of a cycle of length $4$: $\{1,2,3,4\}$ with dimers $\{12,23,34,41\}$.
Then $P=x_1x_2+x_2x_3+x_3x_4+x_4x_1=(x_1+x_3)(x_2+x_4)$.
The space $\Im(D)\subset \R^V$ is the orthocomplement of the vector $(1,-1,1,-1)$, so 
$H_1(T,\R)$ has rank $1$ and is generated by this vector.
The feasible $\vec\alpha$ are those which satisfy $\sum_v\alpha_v=2$ and are in $\Im(D)$, that is,
satisfy the conditions $\alpha_1+\alpha_3=1=\alpha_2+\alpha_4$.
The criticality equations are 
$$\frac{x_1x_2+x_1x_4}{P}=\alpha_1,\frac{x_1x_2+x_2x_3}{P}=\alpha_2, \frac{x_2x_3+x_3x_4}{P}=\alpha_3,
\frac{x_1x_4+x_3x_4}{P}=\alpha_4,$$
which reduce to 
$$\frac{x_1}{x_1+x_3}=\alpha_1,~~\frac{x_2}{x_2+x_4}=\alpha_2.$$
Solutions are not unique: given any solution we can multiply $x_1,x_3$ by a constant $t$ and divide $x_2,x_4$ by $t$ to get another solution.
We have
$$F(X_1,\dots,X_4) = \log((e^{X_1}+e^{X_3})(e^{X_2}+e^{X_4})).$$
This leads to the growth rate 
\begin{align*}\sigma(\vec\alpha) &= -\alpha_1\log\alpha_1-\alpha_2\log\alpha_2-\alpha_3\log\alpha_3-\alpha_4\log\alpha_4 \\
&= -\alpha_1\log\alpha_1-\alpha_2\log\alpha_2-(1-\alpha_1)\log(1-\alpha_1)-(1-\alpha_2)\log(1-\alpha_2)\\
&= h(\alpha_1)+h(\alpha_2)
\end{align*}
where $h(p)$ is the Shannon entropy $h(p)=-p\log p -(1-p)\log(1-p).$

\section{Dimers}

A special case of the multinomial tiling model is the \emph{multinomial dimer model}, where tiles are
simply all pairs of adjacent vertices (also known as ``dimers"). A $1$-dimer tiling is
then a perfect matching, also known as \emph{dimer cover} of $\G$.

\subsection{Bipartite graphs}

For the dimer model, when $\G$ is bipartite, $V=B\cup W$, there is an
equivalent but perhaps more efficient method of computing $Z(w,\N)$. 
By Corollary \ref{probs} we need to look for a gauge function $x:V\to\R_{>0}$
such that 
\be\label{beq}\alpha_\b = \sum_{\w\sim\b}w_{\w\b}x_\w x_\b\ee and likewise for white
vertices. However we can use (\ref{beq}) to define $x_\b$:
$$x_\b = \frac{\alpha_\b}{ \sum_{\w\sim\b}w_{\w\b}x_\w}.$$
Then we only have equations involving the remaining half the variables: those at the white vertices $\x_\w$.
These equations are
\be\label{saddlebip}\alpha_\w = \sum_{\b\sim\w}\alpha_\b\frac{w_{\w\b}x_\w}{\sum_{\w\sim\b}w_{\w\b}x_\w}.\ee

In the standard case where $N_v\equiv N$, all the $\alpha_\w,\alpha_\b$ are equal
and the saddle point equations correspond to the property that the \emph{sum of (critical) edge 
weights at each vertex is $1$}. This is just a restatement of Cororollary \ref{probs} in this 
setting, since the sum of edge probabilities at each vertex is $1$ for a 
random dimer cover.

\subsection{Aztec Diamond Example}

The \emph{Aztec diamond} of order $n$ is a diamond-shaped subregion of $\Z^2$ of horizontal diameter $2n-1$; see Figure \ref{AD}, left panel for the $n=4$ Aztec diamond. It is known to have $2^{n(n+1)/2}$ single-dimer covers (see \cite{EKLP} and \cite{EKLP2}). 
Consider $\N$-fold dimer covers with $N_v\equiv N$ and $w_t\equiv 1$. The critical edge weights sum to $1$ at each vertex,
and have the property that around each square face the weights $a,b,c,d$ satisfy $ac=bd$. 
The critical weights for $n=4$ are shown on the left, and for general $n$ (scaled by $n(n-1)$) on the right
in Figure \ref{AD}.

\begin{figure}
\begin{center}\includegraphics[width=2.7in]{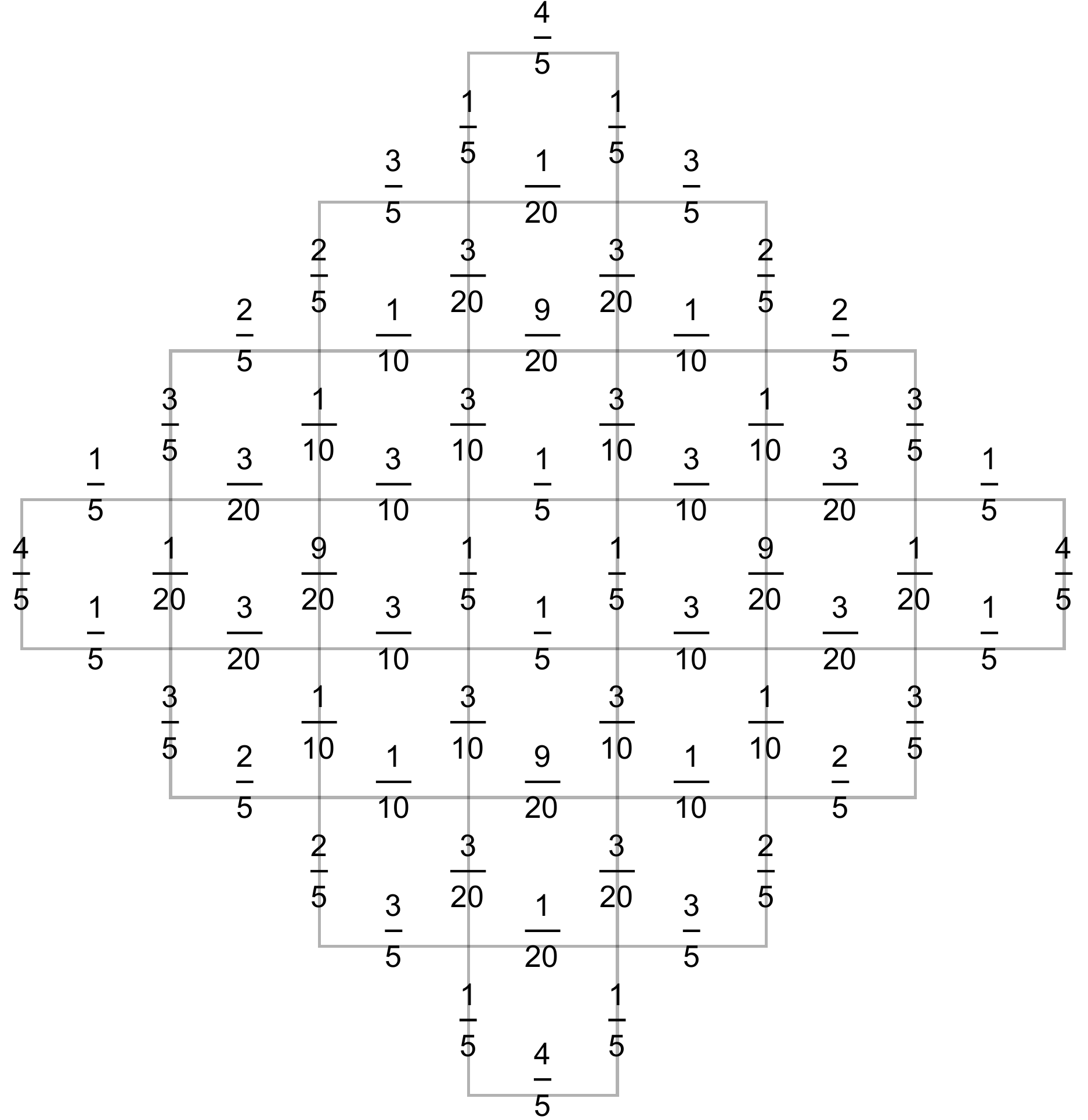}\hskip1cm\includegraphics[width=3in]{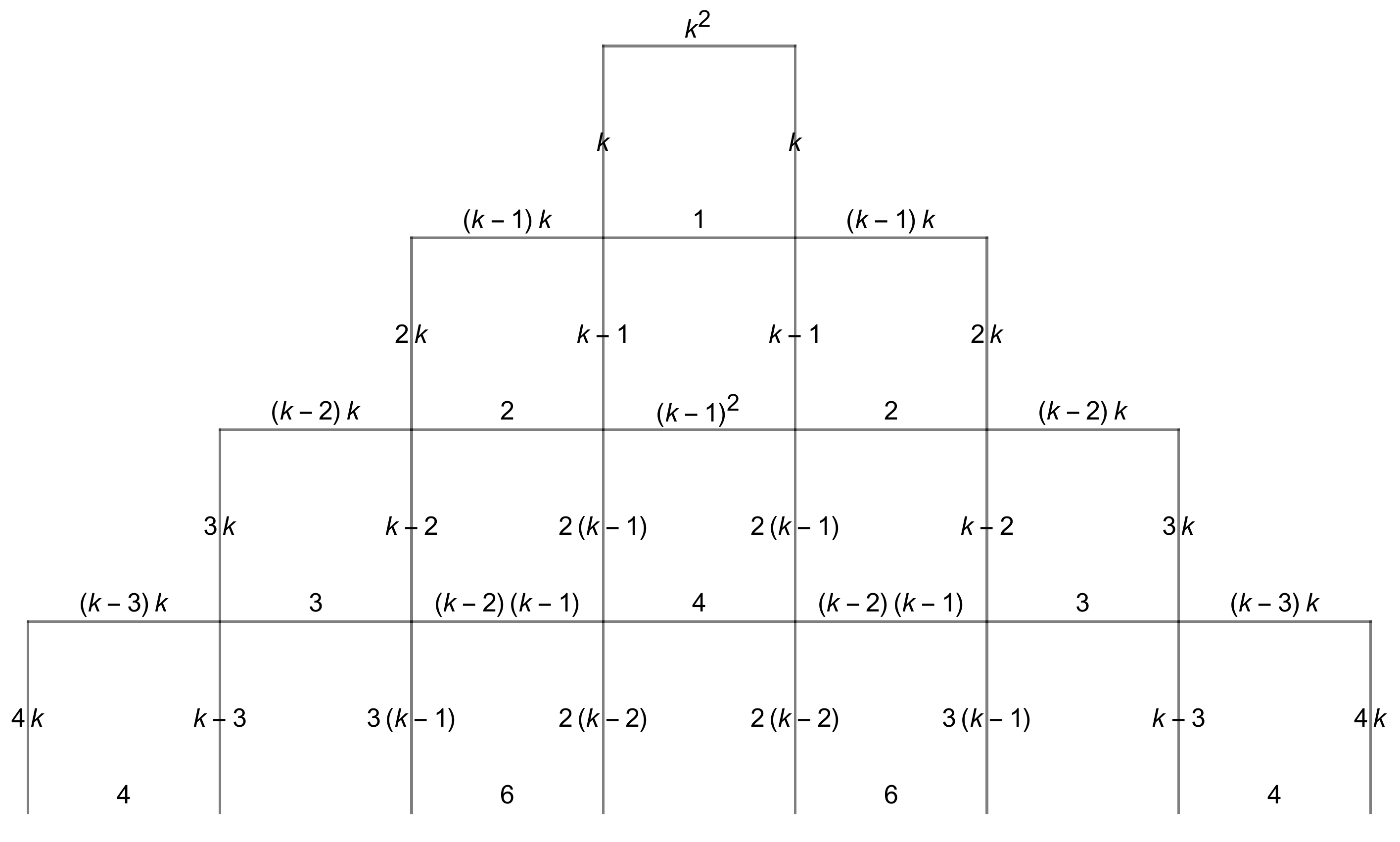}\end{center}
\caption{\small{\label{AD}Aztec diamond of order $4$ (left) in critical gauge. Critical gauge edge weights for general $k$, multiplied by $k(k+1)$ (right). The weight on an edge is a quadratic function of its $x,y$ coordinates and its parity. Out of a white vertex at $x,y$ the edge weights E,N,W,S are
$1/k(k+1)$ times respectively  $(x+y)(x-y), (x+y)(k+1-x+y),(k-x-y)(k+1-x+y),(k-x-y)(x-y)$.}}
\end{figure}

We can work out the growth rate $\sigma$ in this case as follows.
Since $P=1$, and $\alpha_v=\frac{1}{k(k+1)}$ is a constant, we have
$$\sigma=-\frac{1}{k(k+1)}\sum_v\log x_v = -\frac{1}{k(k+1)}\log \prod x_v.$$
This product is the weight of any single dimer cover. The ``all horizontal" dimer cover
has dimers of weight $\frac1{k(k+1)}$ times:  $k^2$ for the top row, $k(k-1)$ and $(k-1)k$ for the next row,
and generally $k(k-i+1),(k-1)(k-i+2),\dots,(k-i+1)k$ for the $i$ row, for $i$ from $1$ to $k$,
then repeating for the bottom half of the diamond. 
The total product of edge weights of the dimer cover is 
$$\frac{(k^k\cdot(k-1)^{k-1}\cdots2^2\cdot 1)^4}{(k(k+1))^{k(k+1)}}$$
This yields for the exponential growth rate the remarkable value
$\sigma=1+O(\frac{\log k}{k^2})$.

Associated to a multinomial dimer cover with constant $N_v\equiv N$ of a subgraph of $\Z^2$  
is a \emph{height function} $h$ on the dual graph. The height function is defined to be zero on a fixed face,
and the change in height across an edge $\w\b$ (when crossing the edge so that the white vertex is on the left)
is $-N/4$ plus the number of dimers on that edge.
In \cite{CEP}, see also \cite{CKP}, the authors prove a limit shape phenomenon for single dimer covers:
the (rescaled by $n$) height function for a random dimer cover of $AD(n)$ converges with probability one as $n\to\infty$
to a nonrandom piecewise analytic function on the rescaled diamond $|x|+|y|\le1$. For the $N$-fold dimer cover discussed above we
get a similar, but analytic, limit shape. It is just the function $h(x,y) = x^2-y^2$.
The general limit shape phenomenon for multinomial tilings is discussed in \cite{KP2}.

\subsection{Path example}

For fixed $n>0$ take $\G$ to be the (bipartite) $n\times n$ honeycomb graph of Figure \ref{2nnhoneycomb}, with edge weights $1$. 
\begin{figure}
\begin{center}\includegraphics[width=2.5in]{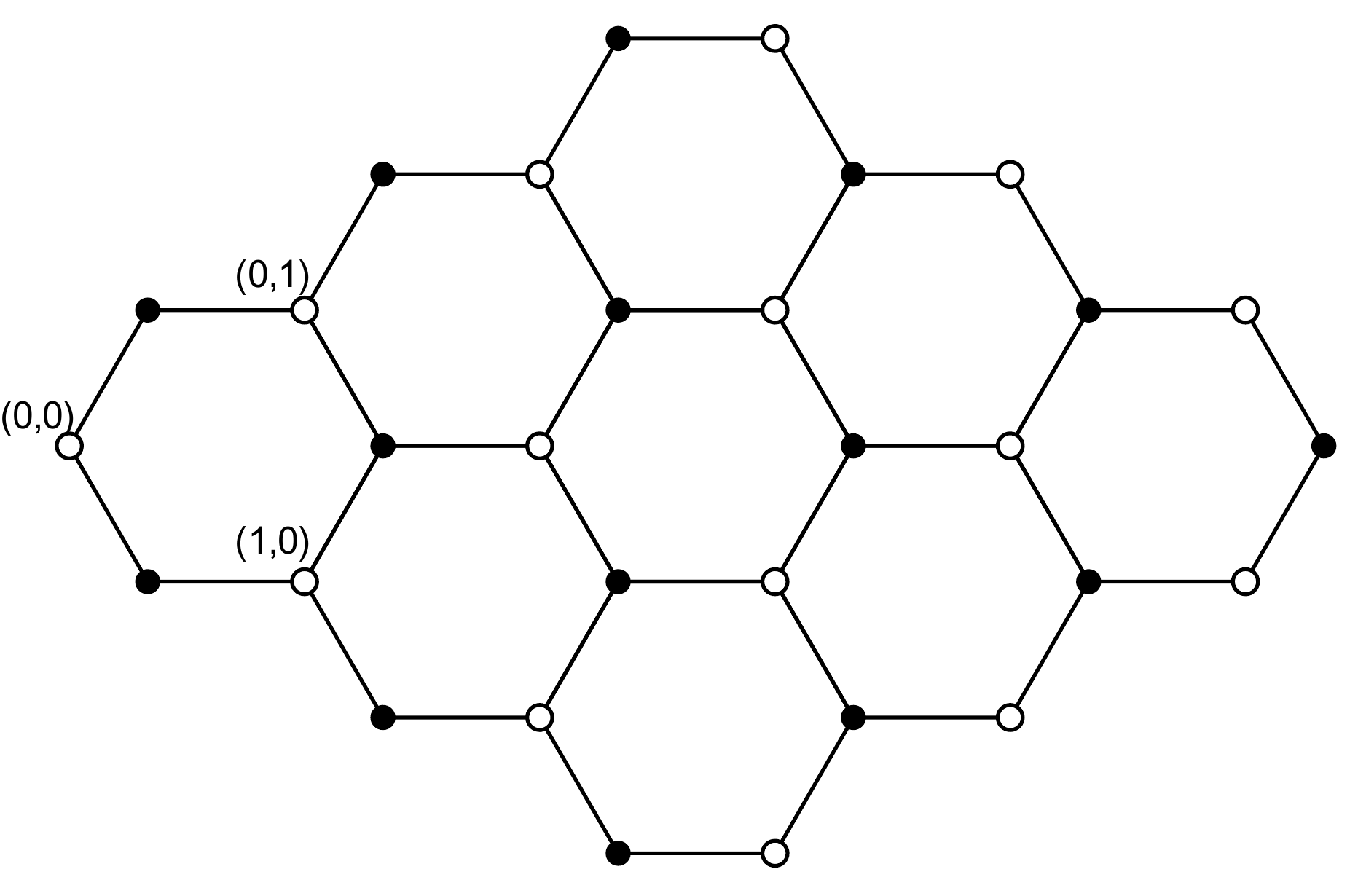}\end{center}
\caption{\small{\label{2nnhoneycomb}$3\times 3$ honeycomb graph. }}
\end{figure}
There are
$\binom{2n}{n}$ dimer covers: dimer covers correspond bijectively to monotone lattice paths from $(0,0)$ to $(n,n)$. The bijection is obtained by taking a dimer cover and shrinking all horizontal edges of $\G$ to points.

Let us consider $\N$-dimer covers of $\G_n$, where $N_v\equiv N$ and $w_t=1$. 
Index the white vertices $x_{i,j}$ as in the figure.
The criticality equations (\ref{saddlebip}) are
\be\label{3sum}\frac{x_{i,j}}{x_{i,j}+x_{i-1,j}+x_{i,j-1}} + \frac{x_{i,j}}{x_{i,j}+x_{i+1,j}+x_{i+1,j-1}}+\frac{x_{i,j}}{x_{i,j}+x_{i,j+1}+x_{i-1,j+1}}=1\ee
and boundary conditions $x_{i,j}=0$ for $i<0$ or $j<0$ or $(i,j)=(n,n)$.

In the limit $n\to\infty$, there is a solution to (\ref{3sum}) given by
$$x_{i,j} = (i+j)! \binom{i+j}{i}.$$
However we don't know if this solution is the only one (since the graph is infinite, unicity does not necessarily hold).

The edge probabilities for this solution are (for horizontal, NE, SE edges respectively out of a white vertex $(i,j)$)
$$\frac{i+j}{i+j+1},~~~ \frac{j+1}{(i+j+1)(i+j+2)},~~~ \frac{i+1}{(i+j+1)(i+j+2)}$$ for the (horizontal, resp. NE, resp. SE)
edge at $(i,j)$.

These edge probabilities give a unit flow on $\NN\times\NN$ from $(0,0)$ to $\infty$, see Figure \ref{NNflow} left panel; the value on an edge represents the flow from left to right along that edge.
\begin{figure}
\begin{center}\includegraphics[width=1.5in]{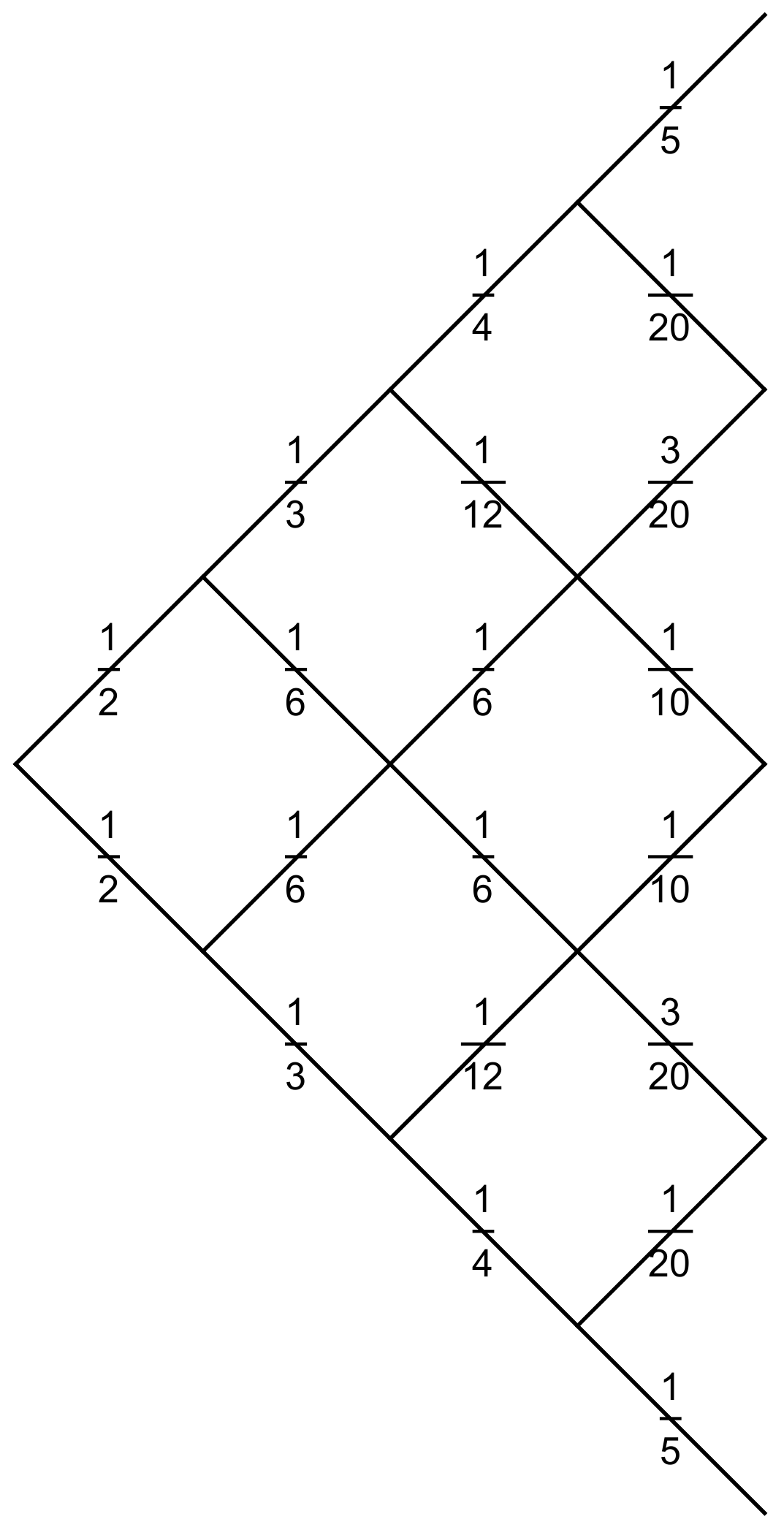}\hskip1cm\includegraphics[width=1.5in]{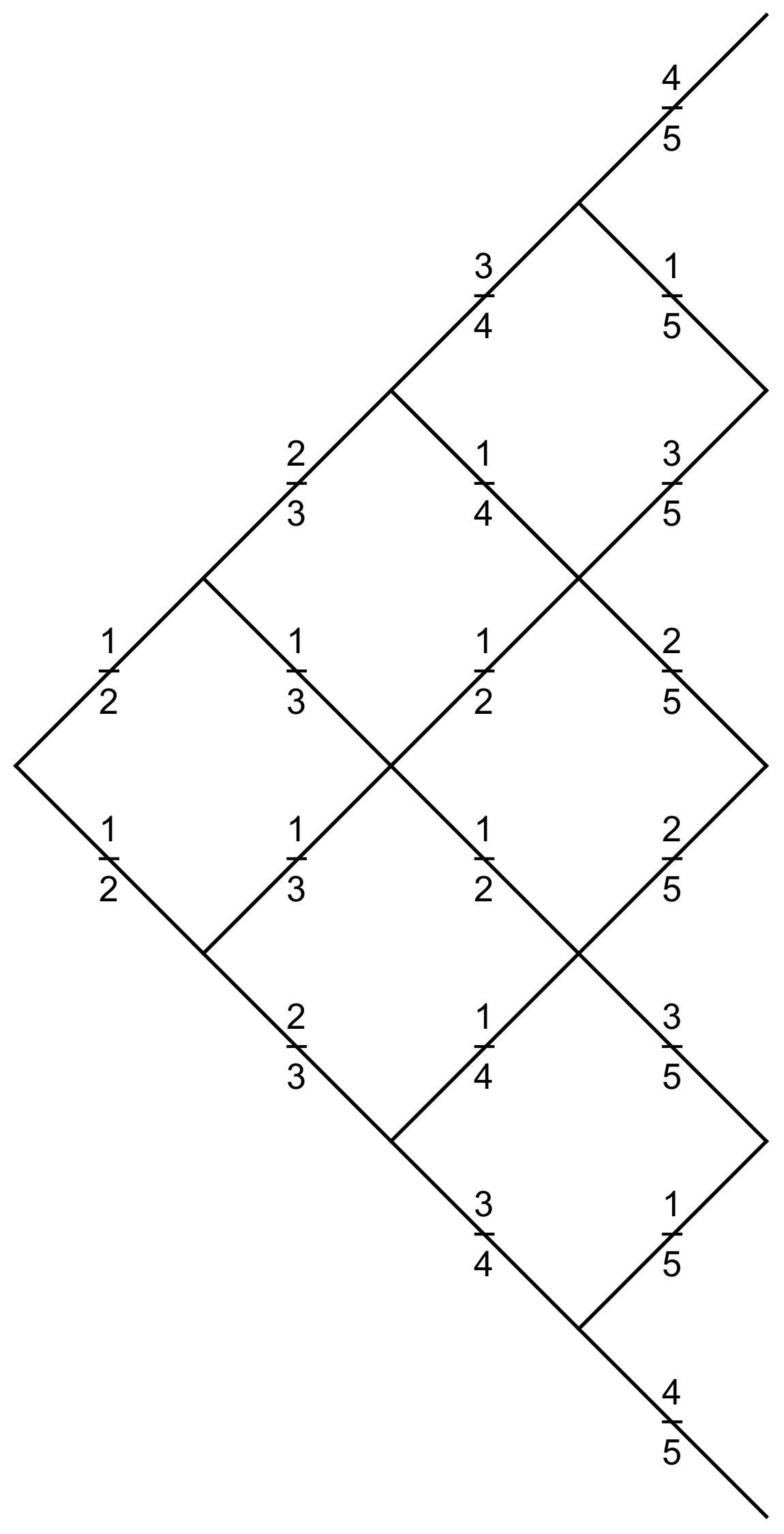}\end{center}
\caption{\small{\label{NNflow}Edge probabilities (left). Corresponding random walk probabilities (right).}}
\end{figure}
The value on an edge is the probability that a certain monotone random walk uses that edge.
The transition probabilities of this random walk are shown on the right in Figure \ref{NNflow};
this random walk is the \emph{Polya urn}\footnote{An urn starts out with one red and one green ball. A ball is selected at random and replaced with another ball of the same color. This process is then repeated many times. The resulting distribution of the number of red balls after $k$ steps is uniform on $[1,k]$.}.

\subsection{Example in higher dimensions}

For a higher dimensional multinomial dimer example, consider the infinite subgraph of $\Z^3$ in the slab $0\le x+y+z\le n$ where $n$ is even,
and $w\equiv 1$. To get a finite graph we can quotient by a cofinite sublattice of $\{x+y+z=0\}$.
Such a graph has a higher proportion of white vertices than black vertices (assuming the origin is white);
the density ratio is $(n+1)/n$.
Take $N_\w$ constant, and $N_\b$ a different constant with $N_\w/N_\b= n/(n+1)$. 
Then a critical gauge is given up to scale by: 
for $\w=(x,y,z),$  $x_{\w\b} = n-x-y-z$ if $b=\w+e_i$ and $x_{\w\b} = x+y+z$
if $\b=\w-e_i$ (here $e_1,e_2,e_3$ are the standard basis vectors).

There are analogous examples in all dimensions $d\ge 1$.

\section{Multiplicities and fluctuations}

\subsection{Changing multiplicities}\label{changemults}

We compute the change in growth rate $\sigma$ (from equation (\ref{sigma})) under a small change in
the multiplicities $\alpha_v\to\alpha_v+d\alpha_v$. This will be used below to compute tile covariances.
Since $\sum_v\alpha_v=\delta$, 
the sum of changes is necessarily zero: $\sum_v d\alpha_v = 0$.

Recall the incidence matrix $D=(D_{v,t})_{v\in V,t\in T}$, defined by 
$D_{v,t}=t_v$.
Differentiating (\ref{simplecriteqns}) we find for each vertex $u$:
$$\sum_{t}w_tt_ux_t(\sum_vt_v\frac{dx_v}{x_v})= d\alpha_u,$$
or
$$\sum_{v,t}D_{u,t}w_tx_tD_{t,v}\frac{dx_v}{x_v} = d\alpha_u.$$
Thus
\be\label{Delta}\sum_{v}\Delta_{u,v}\frac{dx_v}{x_v} = d\alpha_u,\ee
where $\Delta = DCD^*$ is the tiling laplacian for the critical weights.
Equation (\ref{Delta}) says that as a function of $v$, $\frac{dx_v}{x_v}$ is harmonic with respect to the laplacian
$\Delta$ at all vertices $u$ for which $d\alpha_u=0$.

Equation (\ref{Delta}) will have a solution if and only if $d\vec\alpha$ is in the image of $\Delta$, which is the same as 
$\Im(D)$,
since the laplacian is invertible on $\Im(D)$, mapping it to itself (and $\Delta$ is zero on $\ker D^*$). The solution is unique up to an element of
$\ker\Delta = \ker D^*=H_1(T,\R)$; as discussed in Section \ref{gaugesection} these correspond to gauge transformations not changing the tile weights.

\subsection{Covariance of tile densities}\label{covsection}

Let $X_t$ be the random variable counting the number of occurrences of tile $t$ in an $\N$-fold tiling.
We wish to compute the covariance $\Cov(X_t,X_{t'}) = \E[X_tX_{t'}]-\E[X_t]\E[X_{t'}]$
for two tiles $t,t'$.

Note that $X_t$ is itself a sum of $\{0,1\}$-valued random variables, $X_t = \sum_{i=1}^{M_t} X_t^i$, where the sum runs over all $M_t=\prod_{v}\binom{N_v}{t_v}$ possible lifts 
of the tile $t$ to $\G_\N$.
It suffices to compute the covariance $\Cov(X_t^i,X_{t'}^j).$ 

\subsubsection{The case $t\ne t'$.}
Assume first that $t\ne t'$. 
By the symmetry of $\G_\N$, if $t$ and $t'$ are disjoint
this is independent of $i$ and $j$:
$\Cov(X_t^i,X_{t'}^j) = \Cov(X_t^1,X_{t'}^1)$. (If $t,t'$ overlap, see below.)
We have 
\begin{align*}
\E[X_t^1X_{t'}^1] &= \Pr(X_t^1=1,X_{t'}^1=1) \\
&= \Pr(X_{t'}^1=1|X_t^1=1)\Pr(X_t^1=1) \\
&= \Pr\!{}^*(X_{t'}^1=1)\Pr(X_t^1=1),
\end{align*}
where the star denotes the probability measure on the graph $\G_{\N^*}$ where we have reduced the multiplicities of vertices $v$ by $t_v$.
We thus have 
$$\E[X_tX_{t'}] = M_tM_{t'}\E[X_t^1X_{t'}^1] = \E[X_t]\frac{M_{t'}}{M^*_{t'}}\E^*[X_{t'}]=\E[X_t]\E^*[X_{t'}],$$
since $M_{t'}=M^*_{t'}$ when $t,t'$ are disjoint.

When $t$ and $t'$ overlap, the number of disjoint lifts of $t$ and $t'$ is
$$M_{t,t'} := \prod_v\binom{N_v}{t_v,t'_v} =\prod_v\binom{N_v}{t_v}\binom{N_v-t_v}{t'_v}= M_t\prod_{v}\binom{N_v-t_v}{t'_v}.$$
The calculation follows as before but we need to sum over the $M_{t,t'}$ lifts.
We find
$$\E[X_tX_{t'}]  = M_{t,t'}\E[X_t^1X_{t'}^2] = \E[X_t]\E^*[X_{t'}],$$
since the number of lifts of $t'$ in $\G_{\N^*}$ is exactly 
$$\prod_{v}\binom{N_v-t_v}{t'_v}=\frac{M_{t,t'}}{M_{t}}.$$

Thus in either case, if $t\ne t'$,
$$\Cov(X_t,X_{t'}) = \E[X_t](\E^*[X_{t'}]-\E[X_{t'}]).$$

\subsubsection{The case $t=t'$.}
Finally, if $t=t'$ are the same tile, then we have a slightly different computation. Let $X_{t,1}$ and $X_{t,2}$ correspond to disjoint lifts of $t$. With $p=\E[X_{t,1}]$ and 
$p^*=\E^*[X_{t,2}]$, 
\begin{align*}\E[X_t^2] - \E[X_t]^2 &= M_t\E[X_{t,1}]+M_t(M_t-1)\E[X_{t,1}X_{t,2}]-M_t^2\E[X_{t,1}]^2\\
&=M_tp + M_t(M_t-1)pp^*-(M_tp)^2\\
&=\E[X_t](1-p+\frac{M_t-1}{M_t}(\E^*[X_t]-\E[X_t])).
\end{align*}

\subsubsection{Computation of $\E^*$.}
We can now compute $\E^*[X_{t'}]$ from the methods of section \ref{changemults}.
We need to change $\alpha_v=\frac{N_v}{K}$ to $\frac{N_v-t_v}{K-1}$.
Thus
$$d\alpha_v = \frac{\alpha_v}{K-1}-\frac{t_v}{K-1}.$$
Recalling $\E[X_{t'}] = Kw_{t'}x_{t'},$ we have (ignoring lower order terms)
\begin{align}
\E^*[X_{t'}]-\E[X_{t'}] &= (K-1)w_{t'}x_{t'}(1+\sum_{u}t'_u\frac{dx_u}{x_u})-Kw_{t'}x_{t'}\notag\\
&=w_{t'}x_{t'}(-1+(K-1)\sum_{u}t'_u\frac{dx_u}{x_u})\notag\\
&=w_{t'}x_{t'}(-1+(K-1)\sum_{u,v}D_{t',u}\Delta_{u,v}^{-1}d\alpha_v)\notag\\
&=w_{t'}x_{t'}(-1+ \sum_{u,v}D_{t',u}\Delta_{u,v}^{-1} (\alpha_v-D_{v,t}))\notag\\
&=w_{t'}x_{t'}(-1- (D^*\Delta^{-1} D)_{t',t} + \sum_{u,v}D_{t',u}\Delta^{-1}_{u,v}\alpha_v).\label{last}
\end{align}

Now note that 
$$\sum_v\Delta_{u,v} = \sum_{v,t} D_{u,t}w_tx_tD_{t,v}= \delta\sum_{t} D_{u,t}w_tx_t=\delta\alpha_u,$$
so $\Delta^{-1}_{u,v}\alpha_v=\frac1{\delta}1_u$ and thus
$\sum_{u,v}D_{t',u}\Delta^{-1}_{u,v}\alpha_v = 1$. The last sum in (\ref{last}) cancels the $-1$ and we have
$$\E^*[X_{t'}]-\E[X_{t'}] = -w_{t'}x_{t'}\K_{t',t},$$
where $\K=D^*\Delta^{-1}D$,
and so for $t\ne t'$ at critical gauge
$$\Cov(X_t,X_{t'}) =-Kw_tw_t'\K_{t',t}.$$
Likewise when $t=t'$,
$$\Var(X_t) = Kw_t(1-p-w_t\K_{t,t}).$$
As long as $\delta>1$, we can ignore the $p$ for large $K$, and write
\be\label{covform}\Cov(X_t,X_{t'}) = Kw_t(I-C\K)_{t',t}\ee
where $I$ is the identity matrix.
Note that $C\K=(C\K)^2$ is a projection matrix from $\R^T$ onto the subspace $C\Im(D^*)$, with kernel 
$\ker(D)$.
Thus $I-C\K$ is the complementary projection.

\subsection{Gaussian fluctuations}

For the multinomial tiling model with multiplicities $\N$, as before let $X_t$ be the random variable
representing the multiplicity of tile $t$. We consider a limit $K\to\infty$
as in Section \ref{asympsection}. Scaling $P$ to $1$, we can consider $P$
to be the probability generating function (pgf) of a single tile.
Then $P^K$ is the pgf of placing $K$ i.i.d.\! tiles. The tile multiplicities
under this process are Poisson$(Kw_t)$ random variables which tend
in the limit of large $K$ to Gaussian random variables which are
independent except for the constraint that their sum is $K$.
When we impose the constraints on the multiplicities $N_v$
this is an additional linear constraint (which implies the first); 
the resulting random variable is
thus also a joint Gaussian. A (multidimensional) Gaussian is determined by its covariance matrix. 
The covariance can in general be obtained from the 
original covariance matrix and the constraint matrix. 
However in our case we have already computed the covariance matrix above in (\ref{covform}).

\begin{thm}\label{cov} In the limit of large $K$ the joint distribution of the $X_t$
tends to a (multidimensional) Gaussian with mean $\E[\vec X] = K\vec w$
and covariance matrix $\Cov(X_s,X_t) = Kw_s(I - C\K)_{t,s}.$
\end{thm}

Examples are explored in Section \ref{crystalsection}.

\subsection{Coulomb gas}\label{coulomb}

In this section we assume for simplicity that $\G$ is regular, $T$ and $w$ are symmetric under a transitive group of 
automorphisms of $\G$, and $N_v\equiv N$. We also assume $t_v\in\{0,1\}$ for all tiles and vertices. 
We have $K=nN/\delta$. 
Under these conditions the critical weights are $x_v\equiv x$ where $x$ is a constant.  

Suppose that we take a small perturbation of the multiplicities
$N_v = N+q_v$ where $\frac{q_v}{N}=o(1)$. We call $q_v$ the \emph{$d$-charge} at $v$. It can be positive or negative;
we'll take the sum of $d$-charges to be zero. Let us compute $\sigma$ as a function of the $q_v$.

We consider perturbations of $\sigma$ with respect to the multiplicities $\alpha_v$.
Each $\alpha_v$ is of the form $\alpha_v = \delta/n + d\alpha_v$ where $d\alpha_v = \frac{\delta q_v}{nN}$.

To first order from (\ref{sigma}) we have
$$d\sigma(\vec\alpha) = \sum_v\frac{P_{x_v}}Pdx_v - \sum_v\alpha_v\frac{dx_v}{x_v} - \sum_v\log x_v d\alpha_v = - \sum_v\log x_v d\alpha_v = 0$$
since $\log x_v$ is constant.
To second order we have
$$\sigma(\vec\alpha+d\vec\alpha)-\sigma(\vec\alpha) = $$
$$\sum_v(\frac{P_{x_vx_v}}{P}-\frac{P_{x_v}^2}{P^2})\frac{dx_v^2}2   + \sum_{u\ne v}(\frac{P_{x_ux_v}}{P}-\frac{P_{x_u}P_{x_v}}{P^2})dx_udx_v - \sum_v\frac{d\alpha_vdx_v}{x_v}+ 
\sum_v\frac{\alpha_v}{x_v^2}\frac{dx_v^2}{2}.$$
Substituting $x_v=x, P_{x_v}= \alpha_v/x, P=1, P_{x_vx_v}=0$ and $\alpha_v=\delta/n$ gives
$$=\sum_v(\frac{\delta}{n}-\frac{\delta^2}{n^2})\frac{dx_v^2}{2x^2}  + 
\sum_{u\ne v}(\sum_tw_tt_ut_v-\frac{\delta^2}{n^2})\frac{dx_u}x\frac{dx_v}x - 
\sum_v d\alpha_v\frac{dx_v}{x}$$
which we can write as
$$=\frac{d\vec x^t}x\left(\frac12\Delta-\frac{\delta^2}{2n^2}J\right)\frac{d\vec x}x-d\vec \alpha^t\cdot\frac{d\vec x}x$$
where $J$ is the all-$1$'s matrix. But  
$$\frac{d\vec x^t}xJ\frac{d\vec x}x=d\vec \alpha\Delta^{-1}J\Delta^{-1}d\vec \alpha=0$$
(since $\Delta^{-1}$ has constant row and column sums, $\Delta^{-1}J\Delta^{-1}$ is a multiple of $J$).
We are left with 
\begin{thm} Under the above conditions on $\G,\N,T,w$ the second derivative of $\sigma$ in direction 
$\vec\alpha$ (of sum zero) is
$$d^2\sigma=-\frac12d\vec \alpha\Delta^{-1}d\vec \alpha.$$
\end{thm}

If we fix the $d$-charges but allow them to move from vertex to vertex, considering $\sigma$ as
a function of position of these charges,
we see that $d$-charges interact with a potential defined by $\Delta^{-1}$. 
This is the \emph{Coulomb potential} associated with $\Delta$. 

\subsubsection{Example: dimer case}
Consider for example the case in which $\G$ is bipartite, and we have the multinomial dimer model. 
Define the \emph{charge} $\tilde q_v$ of a vertex of d-charge $q_v$ to be
$\tilde q_v=d(v)q_v$ where $d(v) = 1$ for a white vertex and $d(v)=-1$ for a black vertex.
Then note that $\tilde\Delta_{u,v} = d(u)\Delta_{u,v} d(v)$ is the \emph{standard} graph laplacian. 

Suppose we fix the charges $\tilde q_v$, but not their locations.
We can interpret $\sigma$ as a function of position as a potential energy which
causes like charges to repel and opposite charges to attract.
That is, $\sigma$ is larger when like charges are farther apart and opposite charges are closer.
The \emph{force} of repulsion/attraction is naturally given by the gradient of the Green's function 
$\tilde\Delta^{-1}$ (that is, the gradient of the potential). 
For $\G = \Z^2$, for example this is a ``$1/r$'' force for distant particles, where $r$ is the vector between them.
For $\G = \Z^3$ it is a ``$1/r^2$" force for distant particles. 
These conclusions are consistent with standard electrostatics in $\R^2$ and $\R^3$, when the charges are
far apart (compared to the lattice spacing). 
In $2d$ these results also agree with corresponding results obtained for the single dimer model
by Ciucu \cite{Ciucu}.

\section{Crystallization}\label{crystalsection}

In this section we restrict our graphs $\G$ to be subgraphs of $\Z^d$, and we consider tilings with
tiles $T$ which are translates of one or more ``prototiles''.
In the simplest case we have only two prototiles $\t,\t_0$, where $\t_0$ is a \emph{single vertex}. Then a 
$\N$-fold tiling of $\G$ with $T$ is a tiling of $\G_\N$ with lifts of translates of $\t$ which has a number of \emph{holes}  (which are
locations which are covered by lifts of translates of $\t_0$). 
We are interested in what happens when the fractional density of holes goes to zero. 
Depending on the shape of $\t$, the system will \emph{crystallize} (Theorem \ref{crystalthm} below).

Before giving a general result, we work out some explicit examples, which are of interest in their own right,
and which illustrate the general situation. First in dimension one we consider the case of triominos
with no holes (Section \ref{3bars}) and
then with a positive fraction of holes (Section \ref{3barsdefects}), and then the $L$ triomino in $\Z^2$ (Section \ref{L}).

\subsection{Examples in one dimension}
\subsubsection{Example: bars of length $3$}\label{3bars}

Consider tilings of the cycle $\G=\Z/n\Z$ with translates of the triomino $\t=\{0,1,2\}$,
that is $T=\{t_0,\dots,t_{n-1}\}$ with $t_i=\{i,i+1,i+2\}$ with cyclic indices.
We choose constant multiplicities $N_v\equiv N$. Then the total number of tiles is $K=Nn/3$ and 
the fraction of tiles per vertex is $\alpha=3/n$.
Since the graph is regular, for the critical gauge we have $x_v\equiv x,$ and $P=nx^3=1$, so 
the weight per tile is $w_t=x^3=1/n$.
The tile laplacian $\Delta$ satisfies, for a function $f\in\R^V$,
$$(\Delta f)(j) = \frac1n(3f(j)+2f(j-1)+2f(j+1)+f(j+2)+f(j-2))$$ 
with cyclic indices. 

It is convenient to use the Fourier transform to invert $\Delta$. 
For $z$ an $n$th root of $1$ let $U_z$ be the subspace of $\C^V$ consisting of $z$-periodic functions
$$U_z = \{f\in\C^{V}~|~f(x+1)=zf(x)\}.$$
Here $U_z$ is the eigenspace for translation by $1$ on $\G$ with eigenvalue $z$. 
The operator $\Delta$ preserves each $U_z$ and its action on $U_z$ is multiplication by 
$$\lambda_z = \frac1n(1+z+z^2)(1+z^{-1}+z^{-2}).$$

The matrix $D$ can likewise be written in a Fourier basis, if we identify a tile $t_i$ with the location of its left endpoint $i$. 
The action of $D$ on $U_z$
is then multiplication by $1+z+z^2$, and that of $D^*$ is multiplication by $1+1/z+1/z^2$. 
We see that $\K=D^*\Delta^{-1}D$ is simply multiplication by $n$ on subspaces $U_z$ for which $z^2+z+1\ne0$,
and $0$ on subspaces $U_z$ for which $z^2+z+1=0.$

Suppose that $n$ is not a multiple of $3$. Then $\K$ is a scalar, equal to multiplication by $n$.
The covariance matrix is identically zero: $\Cov(X_t,X_{t'}) = 0.$
This is not surprising since
the tile multiplicities $X_i$ are not random: we necessarily have $X_i=N/3$ for all $i$.

Suppose that $n$ is a multiple of $3$. Then $\frac1n\K$ is a projection matrix $\P_3$
onto the span of the subspaces $U_z$
where $z^2+z+1\ne 0$. The covariance matrix is $\Cov(X_t,X_t') = \frac{K}{n}(I-\P_3)$, 
where $I-\P_3$ is the projection onto the (two-dimensional) span of the $U_z$ where $z^2+z+1=0$. 
That is, in the standard basis on $\R^V$,
\be\label{K0}\Cov = \frac{K}{n^2}\begin{pmatrix}2&-1&-1&2\\-1&2&-1&-1\\
-1&-1&2&-1\\
2&-1&-1&2\\
&&&&\ddots\end{pmatrix}.\ee
Pairs of tiles $t,t'$ at distance a multiple of $3$ are perfectly correlated: we have a crystal.
This is also not surprising since the multiplicity at $j$ determines $X_j$ as a function of $X_{j-1},X_{j-2}$:
we have $X_{j-2}+X_{j-1}+X_j = N$, which implies $X_j=X_{j+3}$ for all $j$.

\subsubsection{Example: bars of length $3$ and singletons}\label{3barsdefects}

Let us consider a variant of the above model, where we also allow singleton tiles, with weight $1$. As in Section \ref{number} we include a dummy vertex $v_0$ in $\G$;
we can use the multiplicity $N_0$ of the dummy vertex to control the number of singleton tiles.
We have $N_0+nN=3K$, and
$\alpha_0=N_0/K$ and $\alpha=N/K$ satisfy $\alpha_0+n\alpha=3$.
Also 
$$P_0=\frac{x_0^2}2\sum x_i + \sum x_ix_{i+1}x_{i+2} = \frac{n}2 x_0^2x + nx^3$$
where we replaced $x_i$ with $x$ by circular symmetry.
The critical weights $w_0,w$ are $w_0=\frac{\alpha_0}{2n}$ and $w=\frac{2-\alpha_0}{2n}.$

In this case the laplacian can be partially diagonalized.
We write 
$$\R^\G = \left(\bigoplus_{z^n=1} U_z\right) \oplus U_0$$
wjere $U_z$ is as above and $U_0\cong\R$ is the space of functions on $v_0$.
The laplacian $\Delta$ preserves each $U_z$ for $z\ne 1,0$, and acts by multiplication
by $$\lambda_z = w_0+w(1+z+z^2)(1+z^{-1}+z^{-2})$$ on these $U_z$.
The laplacian does not preserve $U_1$ but does preserve the sum $U_1\oplus U_0$. 
On the space $U_1\oplus U_0$, with basis given by 
$e_1=(1,\dots,1,0)$ and $e_0=(0,\dots,0,1)$, the laplacian acts as the matrix 
\be\label{V10mtx}\begin{pmatrix}w_0+9w&2w_0\\2nw_0&4nw_0
\end{pmatrix}.
\ee
More generally, for a tile of size $A$, the matrix would be 
$$\begin{pmatrix}w_0+A^2w&(A-1)w_0\\(A-1)nw_0&(A-1)^2nw_0
\end{pmatrix};$$
we'll use this below.
Likewise the kernel 
$\K = D^*\Delta^{-1}D$ has a similar decomposition into the subspaces $U_z$ and
$U_1\oplus U_0$. On $U_z$ it acts as multiplication by 
$$\frac{(1+z+z^2)(1+z^{-1}+z^{-2})}{w_0+w(1+z+z^2)(1+z^{-1}+z^{-2})}.$$
On $U_1\oplus U_0$, if we take a single triomino $X_t$ it corresponds to the vector 
$\frac{3}{\sqrt{n}}e_1+ 0e_0$. Thus the contribution to the covariance is 
$\frac1{wn},$ which is $\frac{9}n$ times the $1,1$ entry in the inverse of (\ref{V10mtx}).

Now using Theorem \ref{cov}, 
\be\label{Cov30}\Cov(X_s,X_t) = Kw\delta_{s=t} -\frac{Kw}n - \frac{Kw^2}{n}\sum_{\stackrel{z^n=1}{z\ne 1}}\frac{z^{s-t}(1+z+z^2)(1+z^{-1}+z^{-2})}{w_0+w(1+z+z^2)(1+z^{-1}+z^{-2})}.\ee
Here the second term $-\frac{Kw}n$ is the component of 
$-Kw^2(D^*\Delta^{-1}D)_{s,t}$
on the subspace $U_1$, that is $-Kw^2(\frac1{nw})$.

The covariance (\ref{Cov30}) only depends on $s-t$; without loss of generality assume $t=0$. 
Now (\ref{Cov30}) simplifies to
\be\label{Covsum}
\Cov(X_0,X_s) = \frac{Kw_0}{n}\sum_{\stackrel{z^n=1}{z\ne 1}}\frac{z^s}{w_0/w+(1+z+z^2)(1+z^{-1}+z^{-2})}.\ee

As long as $w_0/w\ll1$, to leading order this sum is localized on the region where $z$ is close to one of the primitive 
cube roots of $1$.
If $\frac{1}{n^2}\ll \frac{w_0}w \ll 1$, we can approximate the sum with an integral.
Writing $z = e^{i(2\pi/3+2\pi j/n)},$ and $\eps=w_0/w$, and letting $u=2\pi j/n$, the contribution near $e^{2\pi i/3}$ is
(up to a $1+o(1)$ factor)
$$\frac{Kw_0\exp(\frac{2\pi is}3)}{2\pi}\int_{-\infty}^{\infty}\frac{e^{i su}du}{\eps+3u^2}=\frac{Kw_0\exp(\frac{2\pi is}3)}{2\sqrt{3\eps}}e^{-|s|\sqrt{\eps/3}}.$$
The integral near the other primitive cube root of $1$ contributes the complex conjugate, and we get
\be\label{covtris}\Cov(X_0,X_s)\sim\frac{Kw_0\cos(\frac{2\pi s}3)}{\sqrt{3\eps}}e^{-|s|\sqrt{\eps/3}}.\ee
We see the onset of the $3$-periodic correlations between the $X_s$ as $\eps\to0$. 

Note that the variances and covariances here are much larger than in the $w_0=0$ case: here they are of order
$$\frac{Kw_0}{\sqrt{\eps}} = K\sqrt{ww_0} = \frac{K\sqrt{\alpha_0(2-\alpha_0)}}{2n},$$
which is (for constant $\alpha_0$) a factor of $n$ larger than that for the pure triomino case above.

For smaller $\eps$, write $\eps=\beta/n^2$; in this case expanding near the primitive cube roots of $1$ gives
\begin{eqnarray*}
\Cov(X_0,X_s)&\sim& \frac{K\beta\cos(\frac{2\pi s}3)}{n^2}\sum_{j\in\Z}\frac{1}{\beta+12\pi^2j^2} \\
&=&
\frac{K\sqrt{\beta}\cos(\frac{2\pi s}3)}{2n^2\sqrt{3}}\coth(\frac{\sqrt{\beta}}{2\sqrt{3}}) = \frac{K\cos(\frac{2\pi s}3)}{n^2}(1+\frac{\beta}{36}+O(\beta^2)),
\end{eqnarray*}
where in the last line we used the Mittag-Leffler expansion for the hyperbolic cotangent function 
$$\pi \coth(\pi z) =\frac{1}{z} + 2 \sum_{n=1}^{\infty} \frac{z}{z^2+n^2}.$$
\old{
The second equality here comes from taking the logarithmic derivative in Euler's product formula for the sine function (see, for example, \cite{SS})
$$\sin(\pi z) = \pi z \prod_{n=1}^{\infty} \left(1 - \frac{z^2}{n^2}\right).$$}

\subsubsection{Quasiperiodic example in dimension $1$}

On $\Z/n\Z$ let $\t$ be the tile with characteristic polynomial $p(z)=2/z+1+2z$. It has two roots of modulus $1$ which are not roots of unity. We tile with $\t$ and $\t_0$ the
singleton, and as before let $\eps=w_0/w$. The covariance function is given by the analogue of (\ref{Covsum}) where the 
denominator is replaced by $w_0/w + |p(z)|^2$. For $1/n^2\ll w_0/w\ll1$ we
get
$$\Cov(X_0,X_s) \sim \frac{Kw_0}{2\pi i}\int_{S^1}\frac{z^s}{\eps+p^2}\frac{dz}{z}.$$
The integral can be localized near the two roots $e^{\pm i\theta_0}$ of $p$.
Letting $z=e^{i(\theta_0+x)}$ we get 
$$\Cov(X_0,X_s) \sim 2Kw_0\Re\left[\frac{e^{is\theta_0}}{2\pi}\int_{\R}\frac{e^{isx}dx}{\eps+2p'(\theta_0)^2x^2}\right] = \frac{Kw_0\cos(s\theta_0)e^{-|s|\sqrt{\eps a}}}{\sqrt{\eps a}}$$
where $a=2p'(\theta_0)^2 = 32\sin^2(\theta_0).$

For a similar example without multiplicity, we can take $p(z) = 1+z+z^3+z^5+z^6$.

\subsection{Examples in higher dimensions}

\subsubsection{Example: $L$ triomino}
\label{L}

For a $2d$ example, consider tilings of the torus $\Z^2/n\Z^2$ with translates of the triomino 
$\t=\{(0,0),(1,0),(0,1)\}$,
and constant multiplicities $\N\equiv N$. Let $\Gamma$ be the sublattice of $\Z^2$ generated by $(1,1)$ and $(3,0)$. Translates of $\t$ by $\Gamma$ tile
$\Z^2$. Translates of this tiling by the three cosets of $\Gamma$ in $\Z^2$ give three distinct periodic tilings.

We have $K=Nn^2/3$, $\alpha=3/n^2$.
The graph is regular, so for the critical gauge we have $x_v\equiv x,$ and $w_t=1/n^2$. 

For $z,u$ two $n$th roots of $1$, define 
$U_{z,u}$ to be the subspace of $\C^{V}\cong\C^{n^2}$ consisting of $(z,u)$-periodic functions,
that is, functions $f:\Z^2/n\Z^2\to\C$ such that $f(i+1,j) =z f(i,j)$ and $f(i,j+1)=u f(i,j)$. 

The action of $\Delta$ on $U_{z,u}$ is multiplication by
$$\lambda_{z,w} = \frac1{n^2}(1+z+u)(1+z^{-1}+u^{-1}).$$
If we index tiles according to the location of their lower left vertex, then 
the matrix $D$ corresponds to multiplication by $1+z+u$, and $D^*$ by $1+z^{-1}+u^{-1}$.
Then $\K=D^*\Delta^{-1}D$ is multiplication by $n^2$ of spaces $U_{z,u}$ for which $1+z+u\ne 0$,
and zero on spaces $U_{z,u}$ where $1+z+u=0$. 

If $n$ is not a multiple of $3$, then there are no subspaces $U_{z,u}$ where $1+z+u=0$, and so
$\K=n^2I$ is a scalar multiple of the identity, and $\Cov$ is the zero matrix:
all tile multiplicities are determined and constant.
If $n$ is a multiple of $3$, then $\Cov = \frac{K}{n^2}(I-\P)$ where $I-\P$ is the projection onto the span of 
$U_{\omega,\omega^2}\oplus U_{\omega^2,\omega}$ (here $\omega=e^{2\pi i/3}$). 
The covariance between tiles is $2K/n^4$ if they
lie in the same translate of $\Gamma$ and $-K/n^4$ if they do not. Again this is a \emph{perfect crystal}; the tiles differing
by translates in $\Gamma$ are perfectly correlated. 

Let us now allow a small fraction of singleton tiles: $N_0+n^2 N=3K$, and $\alpha_0=N_0/K$.
As before $w_0=\frac{\alpha_0}{2n^2}$ and $w=\frac{2-\alpha_0}{2n^2}$. 
The laplacian $\Delta$ still preserves each $U_{z,u}$ for $(z,u)\ne(1,1)$, and on $U_{z,u}$ acts by multiplication by
$$\lambda_{z,w} = w_0+w(1+z+u)(1+z^{-1}+u^{-1}).$$
On $U_{1,1}\oplus U_0$ it acts as the matrix
$$\begin{pmatrix}w_0+9w&2w_0\\2n^2w_0&4n^2w_0
\end{pmatrix}.$$
We have
\begin{align*}
\Cov(X_{(0,0)},X_{(s,t)})
&= Kw\delta_{s=t=0} - \frac{Kw}{n^2}-\frac{Kw^2}{n^2}\sum_{\stackrel{z^n=1=u^n}{(z,u)\ne(1,1)}}\frac{z^su^t(1+z+u)(1+z^{-1}+u^{-1})}{w_0+w(1+z+u)(1+z^{-1}+u^{-1})}\\
&= \frac{Kw_0}{n^2}\sum_{\stackrel{z^n=1=u^n}{(z,u)\ne(1,1)}}\frac{z^su^t}{w_0/w+(1+z+u)(1+z^{-1}+u^{-1})}.\end{align*}

Now assume $w_0/w\ll 1$ and is fixed as $n\to\infty$. The sum is then localized to the region near $(z,u) = (\omega,\omega^2)$ or $(\omega^2,\omega).$ 
Near the first point write $z=e^{i(2\pi/3+x)}$ and $u=e^{i(4\pi/3+y)}$, and take $\eps=w_0/w$.
The contribution near this point is 
$$\sim \frac{Kw_0\exp(\frac{2\pi i(s+2t)}3)}{(2\pi)^2}\iint_{\R^2}\frac{e^{i(sx+ty)}dx\,dy}{\eps+x^2-xy+y^2}.$$
The integral here for $(s,t)\ne(0,0)$ is a Bessel-K function (see the appendix Section \ref{Besselscn}), 
$$=\frac{Kw_0\exp(\frac{2\pi i(s+2t)}3)}{\pi\sqrt{3}}B(\frac{2}{\sqrt{3}}\sqrt{\eps(s^2+st+t^2)}).$$
Summing over both roots, and taking $\eps$ small 
we have 
$$\Cov(X_{(0,0)},X_{(s,t)})= \frac{Kw_0\cos(\frac{2\pi}{3}(s+2t))}{\pi\sqrt{3}}\left(\log\frac{1}{\eps}-\log\frac{2(s^2+st+t^2)}{3}-2\gamma_E+O(\eps)\right)$$
where $\gamma_E$ is the Euler gamma.  
 
For the variance, that is, when $(s,t)=(0,0)$,  
$$\Var(X_{0,0}) = \frac{Kw_0}{(2\pi i)^2}\iint\frac{dz du}{(1+u)(z-r_1)(z-r_2)}$$
where $r_1,r_2$ are the roots of the denominator which is a quadratic polynomial in $z$. 
Let $r_1$ be the root inside the unit circle; then using residues this is  
$$=\frac{Kw_0}{2\pi i}\int_{S^1}\frac{du}{(1+u)(r_1-r_2)}$$
$$=\frac{Kw_0}{2\pi}\int_0^{2\pi}\frac{d\theta}{\sqrt{3+6\eps+\eps^2+(4+4\eps)\cos\theta+2\cos(2\theta)}}$$
and splitting the integral into parts near $\theta=2\pi/3,4\pi/3$ and the remainder, we arrive at
$$=\frac{Kw_0}{\pi\sqrt{3}}(\log\frac{9}{\eps})(1+o(1)).$$

See Figure \ref{Lcov} for a plot of the covariances for small $\eps$. 
\begin{figure}
\begin{center}\includegraphics[width=2.5in]{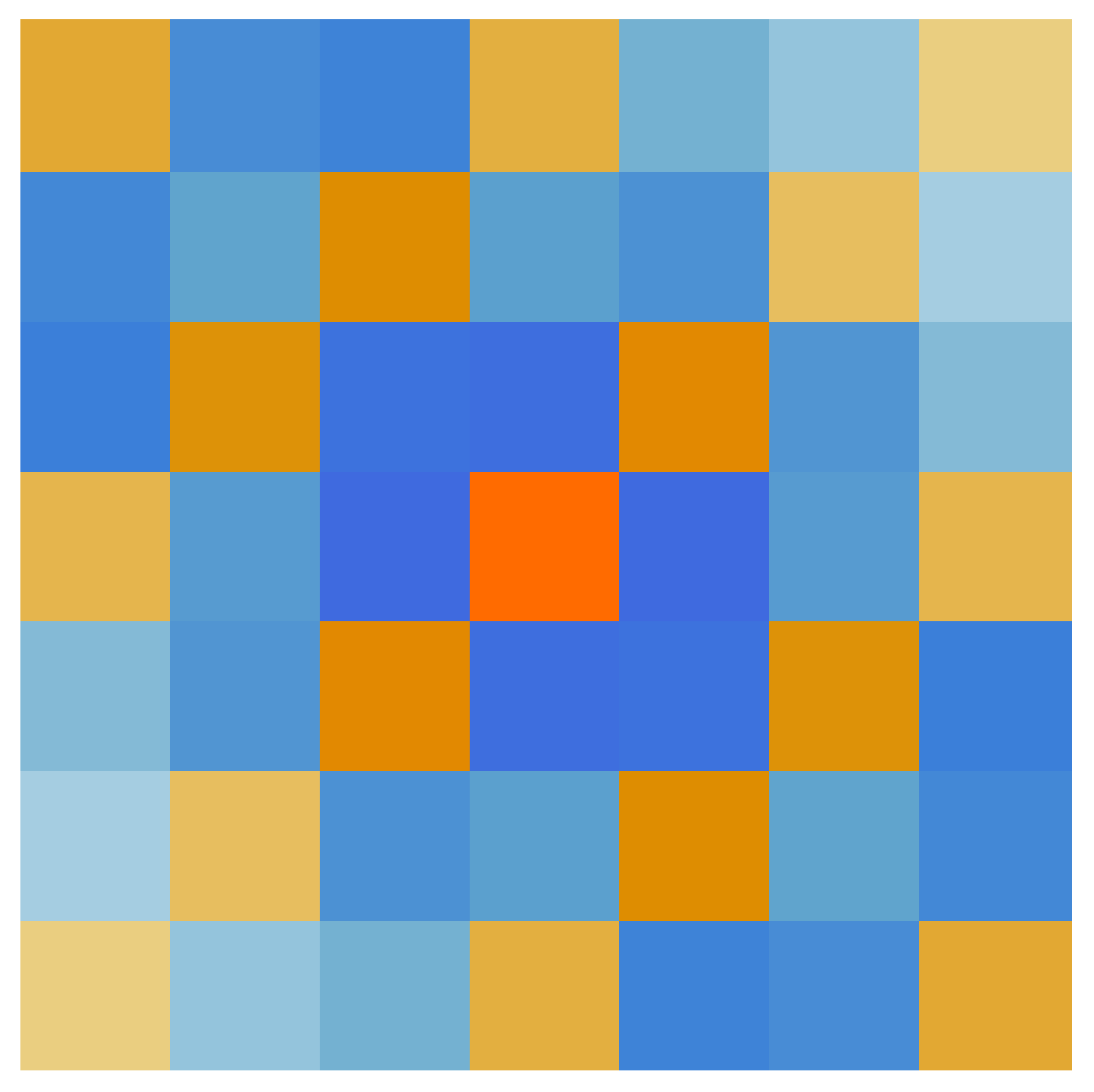}\hskip1cm\includegraphics[width=2.5in]{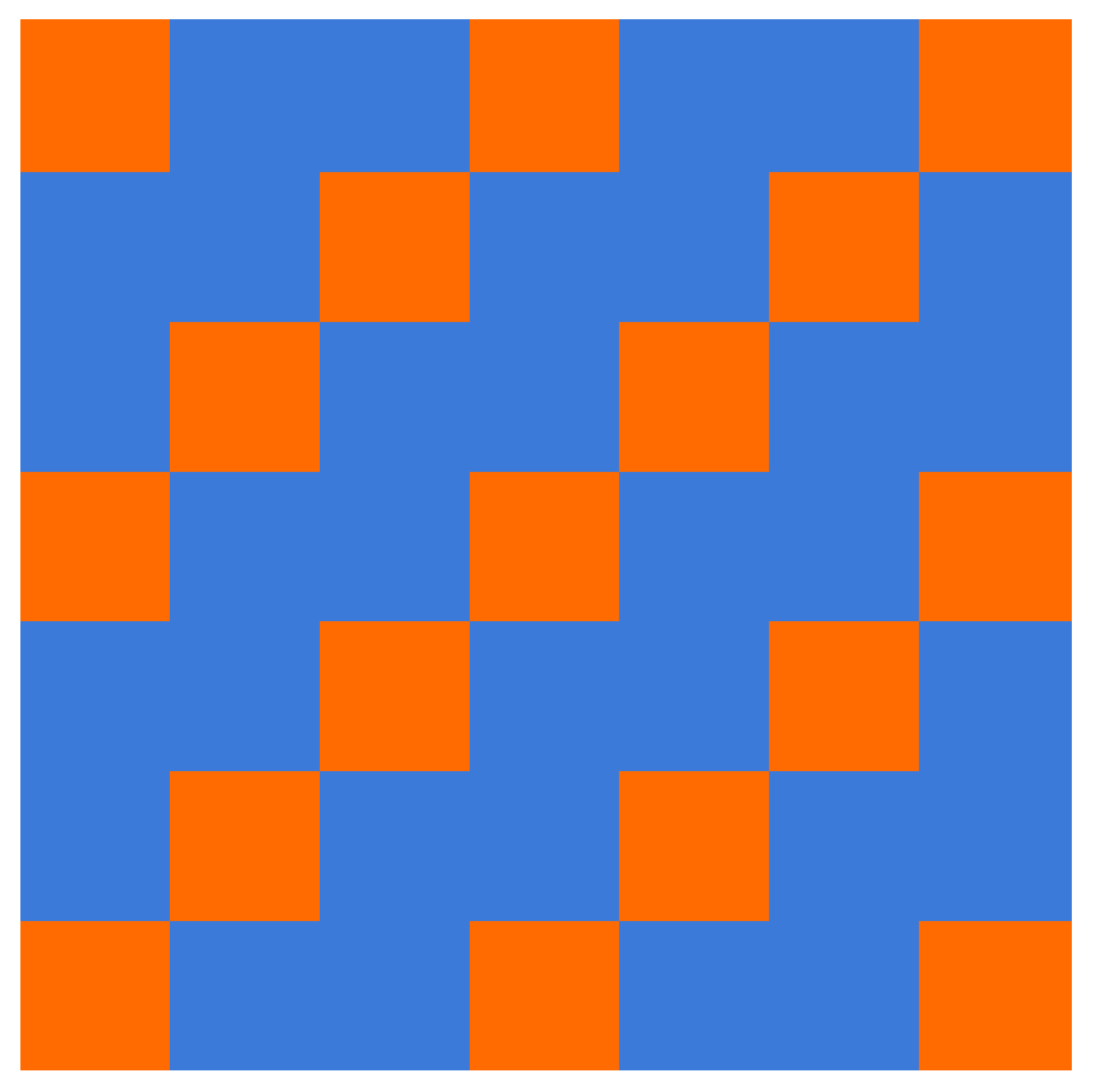}\end{center}
\caption{\small{\label{Lcov}Covariances for the L polyomino for $\eps=0.001$ (left) and in the limit $\eps=0$ (right). Here values are scaled to the range from
$1$ (orange) to $-1$ (dark blue) with $0$ being white.}}
\end{figure}

\subsubsection{Examples in $\Z^2$ with simple zeros.} 

The above example of the $L$ triomino can be generalized.
Consider the case of tiling with a polyomino $\t$ in $\Z^2$ and a small density of singletons $\t_0$.
The characteristic polynomial $p(z,u)$ of a polyomino $\t$ is defined as 
$$p(z,u) = \sum_{(i,j)\in\t} z^iu^j.$$

As in the previous section we have
\begin{align}\nonumber\Cov(X_{0,0},X_{s,t}) &= \frac{Kw_0}{n^2}
\sum_{\stackrel{z^n=1=w^n}{(z,u)\ne(1,1)}}
\frac{z^su^t}{w_0/w+p(z,u)p(1/z,1/u)}\\
\label{Covintegral}&\sim\frac{Kw_0}{(2\pi)^2}\iint_{(S^1)^2} \frac{z^su^t}{\eps+p(z,u)p(1/z,1/u)}\frac{dz}{iz}\frac{du}{iu}.
\end{align}

Suppose now that $p(z,u)$ has a finite number of roots $\{(z_i,u_i)\}_{i=1,\ldots,k}$ on $\T^2=S^1\times S^1$.
A root $(z,u)$ is \emph{simple} if $\frac{zP_z}{uP_u}\not\in\R$; this means the zero set of $p$ intersects
the unit torus transversely at that point. 
We suppose for the moment that all roots $(z_i,u_i)$ are simple.

For $w_0/w$ a small constant the integral (\ref{Covintegral}) can be localized near each root.
Near a simple root $(z_0,u_0)$ the integral is approximated by
$$\frac{Kw_0z_0^su_0^t}{(2\pi)^2}\int_{\R^2} \frac{e^{i(sx+ty)}\,dx\,dy}{\eps+ax^2+bxy+cy^2}$$
where the quadratic form in the denominator is 
$$ax^2+bxy+cy^2 = |z_0P_z(z_0,u_0)x + u_0P_u(z_0,u_0)y|^2.$$
Since the root is simple this quadratic form is positive definite,
and the integral is defined and finite (for $(s,t)\ne(0,0)$). 
This integral is a (linear image of a) Bessel-K function (see section \ref{Besselscn}).
Summing over all roots, the covariance is a superposition
of these Bessel-type functions. 

We also need the $s=t=0$ case, that is, the variance, which we cannot get using Bessel functions.
We have 
$$\Var(X_{0,0}) = \frac{Kw_0}{(2\pi)^2}\iint \frac{1}{\eps+p(z,u)p(1/z,1/u)}\frac{dz}{iz}\frac{du}{iu}.$$
Localize to a ball $B$ of radius $M\sqrt{\eps}$ around a simple root $(z_0,u_0)$i (where $M$ is a large constant). Let $z=z_0e^{is\sqrt{\eps}}$ and $u=u_0e^{it\sqrt{\eps}}$.
The contribution from this ball is
$$\frac{Kw_0}{(2\pi)^2}\int_{B}\frac{ds\,dt}{1+as^2+bst+ct^2+O(\eps^{1/2})} = \frac{Kw_0}{2\pi}\left(\frac{1}{\sqrt{4ac-b^2}}\log\frac{1}{\eps}+O(1)\right).$$
This is then summed over all roots.

\begin{thm}\label{crystalthm} Suppose $\t$ is a polyomino whose characteristic polynomial has only simple roots 
$\{(z_j,u_j)\}_{j=1,\ldots,k}$ on $\T^2$,
and consider tilings of $\Z^2$ with translates of $\t$ and the singleton $\t_0$. Fix a small $\eps=w_0/w$,
then the covariance function is, up to a $1+o_{\eps}(1)$ factor,
$$\Cov(X_{0,0},X_{s,t})= Kw_0\sum_{j=1}^{k} \frac{z_j^su_j^t}{(2\pi)^2}\int_{\R^2}\frac{e^{i(sx+ty)}\,dx\,dy}{
\eps+|z_jP_z(z_j,u_j)x+u_jP_u(z_j,u_j)y|^2}.$$
\end{thm}

\subsubsection{Quasiperiodic example in dimension $2$}\label{qp2}

This is an example with simple roots which are not roots of unity. 
Consider the ``key" polyomino (Figure \ref{qp}) with $$p(z,u) = 1+z+z^2+z^3+z^4+z^5+z^6+u(1+z+z^4+z^6).$$
It has $2$ simple roots on $\T^2$, $(z_0,u_0),(\bar z_0,\bar u_0)$,
and $z_0,u_0$ have arguments $\theta,\phi$ which are not rationally related to each other, that is,
there are no integers $m_1,m_2$ such that $z_0^{m_1}u_0^{m_2}=1$ except $m_1=0=m_2$ 
(this requires a short Galois theory argument).

\begin{figure}[H]
\begin{center}\includegraphics[width=1.5in]{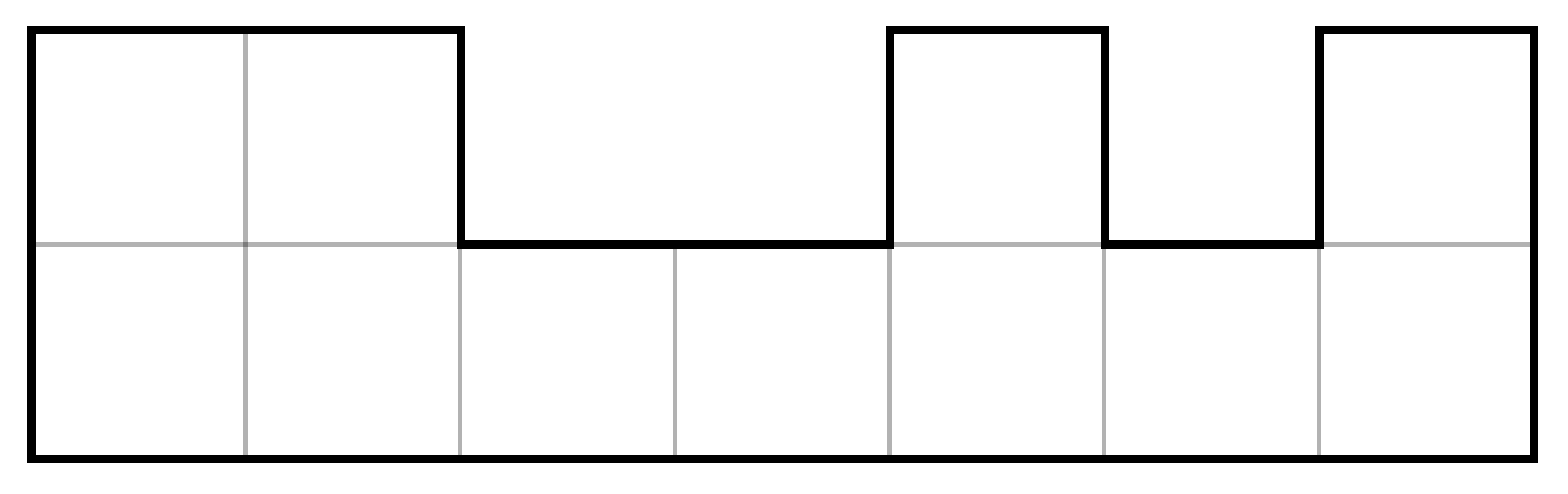}\hskip1in\includegraphics[width=2.in]{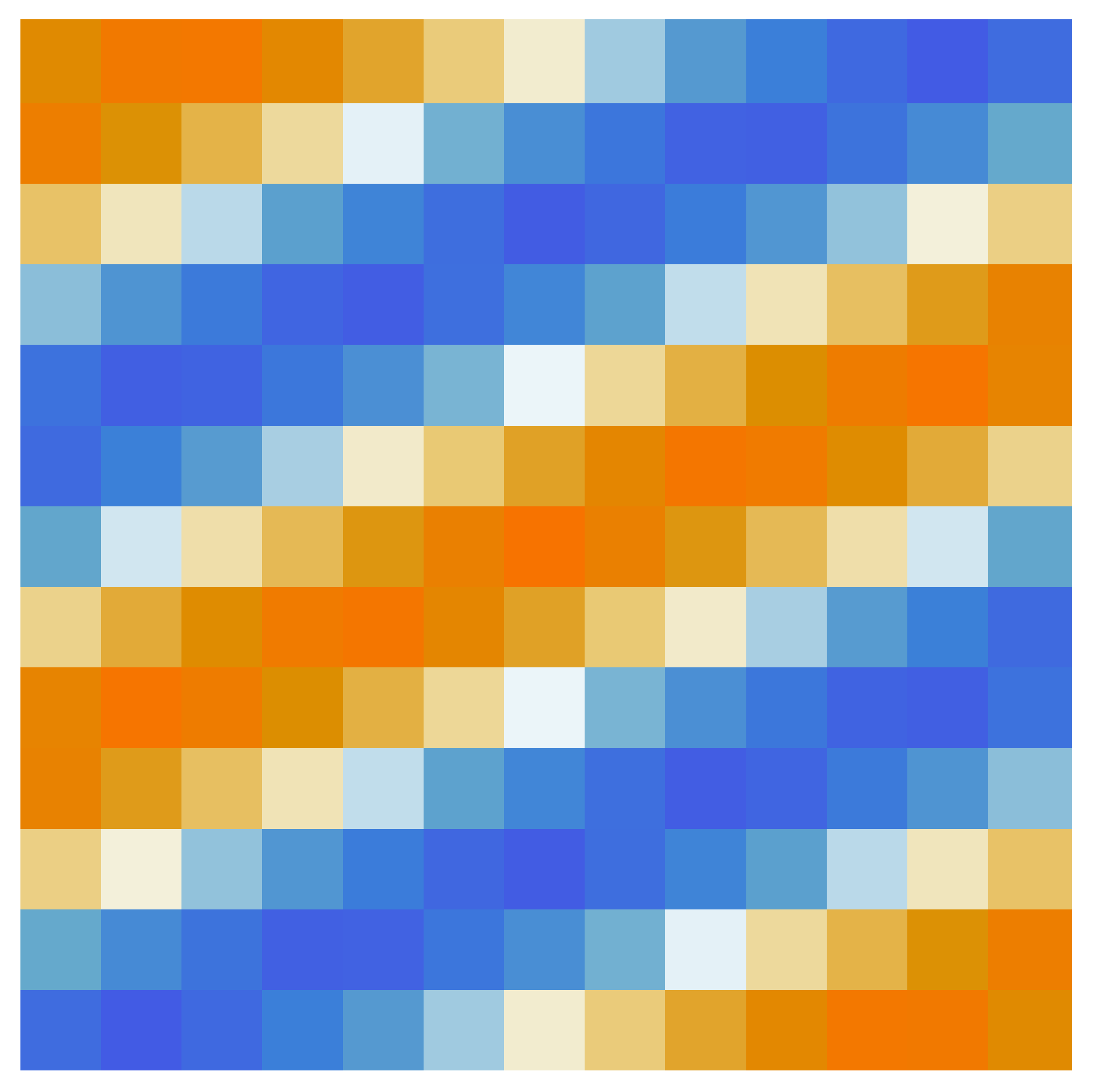}\end{center}
\caption{\small{\label{qp}The key polyomino and its (quasiperiodic) covariance function in the limit $\eps\to0$.}}
\end{figure}

This implies that the resulting covariance function
$\Cov(X_{0,0},X_{s,t})$ is
a quasiperiodic function of $(s,t)$ in the $\eps\to0$ limit: specifically, to leading order 
\be\label{qpcov}\Cov(X_{0,0},X_{s,t})=  Kw_0C(\log\frac{1}{\eps})\cos(s\theta+t\phi)\Big(1+o_\eps(1)\Big)\ee
for a constant $C$. See Figure \ref{qp} for a numerical plot.

A simpler quasiperiodic example, if we allow multiplicities, is the tile with $p(z,w) = 2+3z+4u.$  
It also has $2$ simple roots (with arguments corresponding to the angles of the $2,3,4$-triangle).
The covariance formula is again of the form (\ref{qpcov}) where $\theta,\phi$ are these angles. 

\begin{figure}
\begin{center}\includegraphics[width=2in]{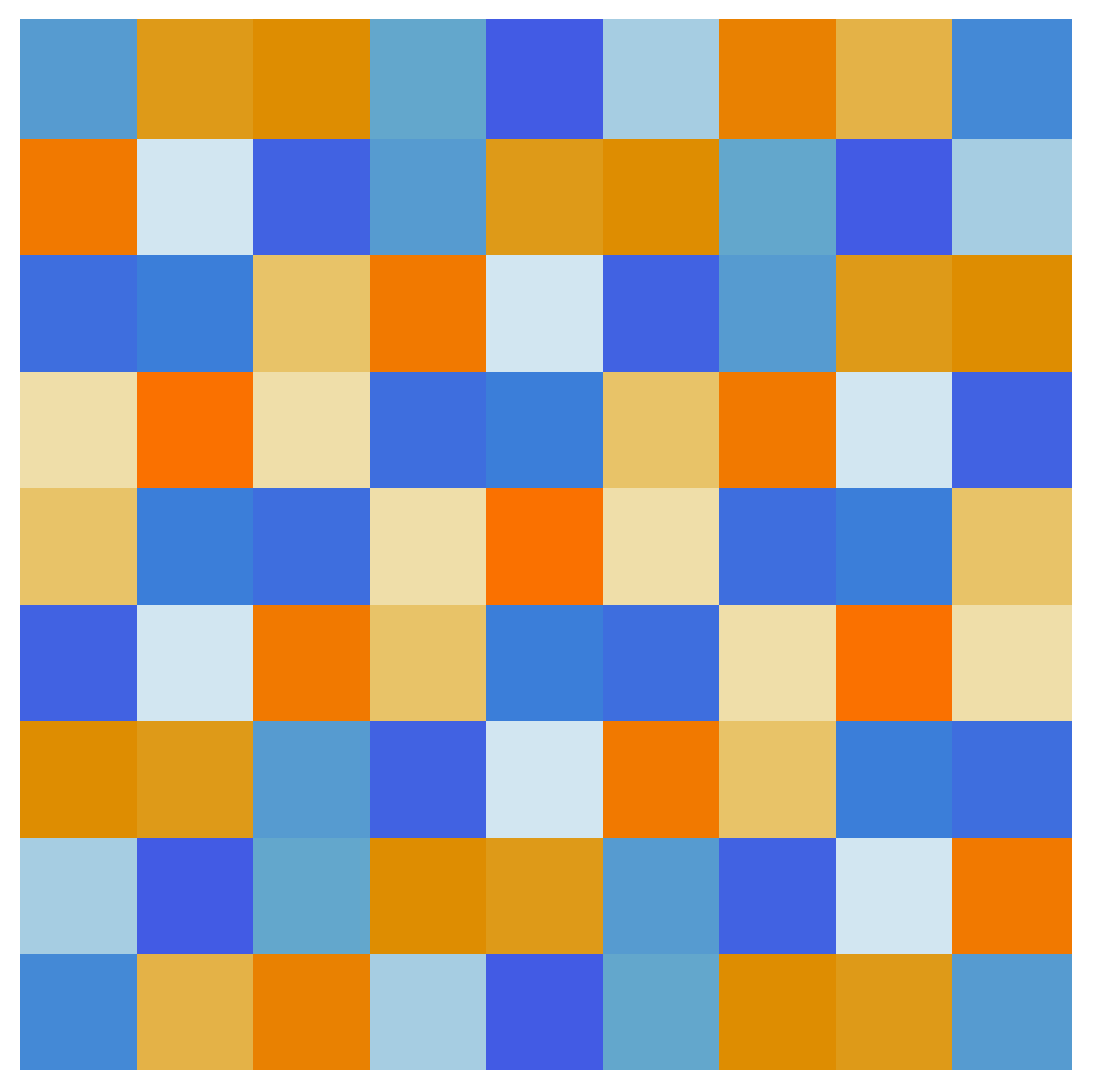}\end{center}
\caption{\small{\label{qptri}Quasiperiodic covariance function (when $\eps=0$) for the $2,3,4$-weighted $L$ polyomino.}}
\end{figure}

\subsection{Other examples}\label{other}

Not all polyominos have roots on $\T^2$. For example 
the tile of Figure \ref{product} has the property that its characteristic polynomial has no roots on $\T^2$.

\begin{figure}[H]
\begin{center}\includegraphics[width=1.in]{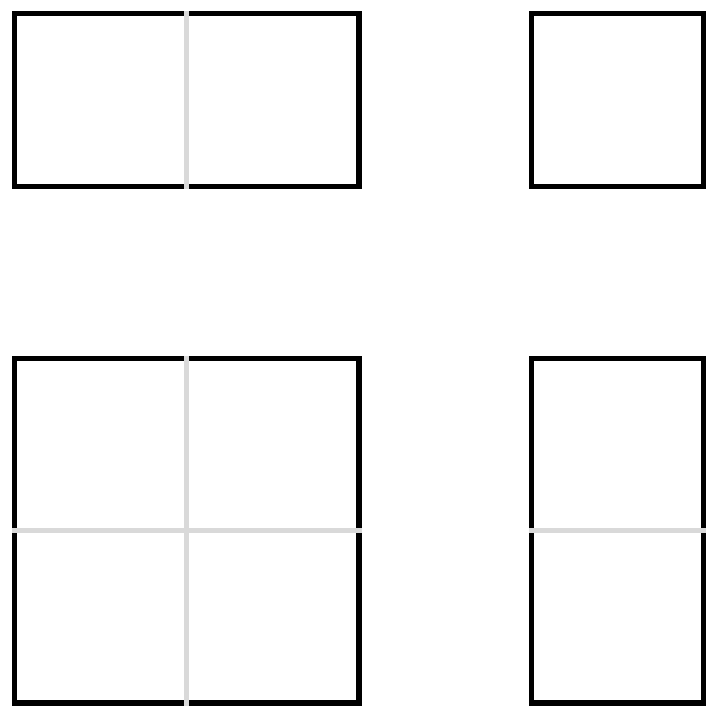}\end{center}
\caption{\small{\label{product}A (nonconnected) polyomino with no crystal structure.}}
\end{figure}

As a consequence its covariance function $\Cov(X_{0,0},X_{s,t})$ decays exponentially in $|s|+|t|$ even
for $w_0=0$. 
This example is however somewhat special since its characteristic polynomial 
$p(z,u)=(1+z+z^3)(1+u+u^3)$ is a product of two $1d$ polynomials. A genuinely 2d example which
does not factor is not easy to find. Here is one:
$$p(z,u) =1 + z + z^2 + z^3 + z^7 + z^9 + z^{12} + z^{13} + z^{17} + u.$$
For polyominos with multiplicity, an easy example is the one with $p(z,u) = 3+z+u$. 

The third class of polyominos in $\Z^2$ has characteristic polynomials with roots on $\T^2$ which are either 
\emph{not simple} or \emph{not isolated}.
And generally for $d$-dimensional polyominos with $d>2$ the roots on $\T^d$ are not typically isolated. 
It is harder
to formulate a general theory encompassing all these cases. We will simply illustrate with a few examples.

\subsubsection{The square polyomino}
We consider the square polyomino with $p(z,u) = (1+z)(1+u).$
We evaluate the integral (\ref{Covintegral}).
We first perform a contour integral over $u$. Assume $t\ge 0$.
\begin{align*}
\frac{Kw_0}{(2\pi i)^2}\int \frac{z^s u^t}{\eps+(1+z)(1+u)(1+\frac1z)(1+\frac1u)}\frac{du}{u}\frac{dz}{z}
&=\frac{Kw_0}{(2\pi i)^2}\int \frac{z^{s} u^{t}}{\eps zu+p^2}\,du\,d z\\
&=\frac{Kw_0}{(2\pi i)^2}\int\frac{z^s u^t}{(1+z)^2(u-r_1)(u-r_2)}du\,dz.
\end{align*}

Roots $r_1,r_2$ of the denominator are real with product $1$; choose $|r_1|<1<|r_2|$.
We get
$$=\frac{Kw_0}{2\pi}\int_0^{2\pi}\frac{z^s r_1^{t}}{(1+z)^2(r_1-r_2)}d\theta =\frac{Kw_0}{2\pi\sqrt{\eps}}\int\frac{\cos(\theta s) r_1^{t}}{\sqrt{8+\eps+8\cos\theta}}d\theta. 
$$
This is an elliptic function. For small $\eps$ the integral concentrates
near $\theta=\pi$,  and is to leading order, for constant $s,t$, 
$$Kw_0\frac{(-1)^{s+t}}{2\pi\sqrt{\eps}}\left(\log\frac{16}{\sqrt{\eps}}-2h_s-2h_t+O(\eps)\right)$$
where $h_m = 1+\frac13+\frac15+\dots+\frac1{2m-1}.$

We thus have 
$$\Cov(X_{0,0},X_{s,t}) = Kw_0\frac{(-1)^{s+t}}{2\pi\sqrt{\eps}}\left(\log\frac{16}{\sqrt{\eps}}-2h_{|s|}-2h_{|t|}+O(\eps)\right).$$
For larger $(s,t)$, on the order $s,t=O(\eps^{-1/2})$, the Fourier coefficients decay exponentially at a rate
determined by the component around $(0,0)$ in the complement of the amoeba of $\eps+|p|^2$,
(the \emph{amoeba} is the image of the zero set of $\eps+|p|^2$ under the map $(z,w)\mapsto(\log|z|,\log|w|)$).
In this case the component around the origin for small $\eps$ tends to a square of width 
$\frac12\sqrt{\eps}$, so the Fourier coefficients are of modulus of order 
$\exp(-\frac{\sqrt{\eps}}2\min\{|s|,|t|\}).$

\begin{figure}
\begin{center}\includegraphics[width=2.5in]{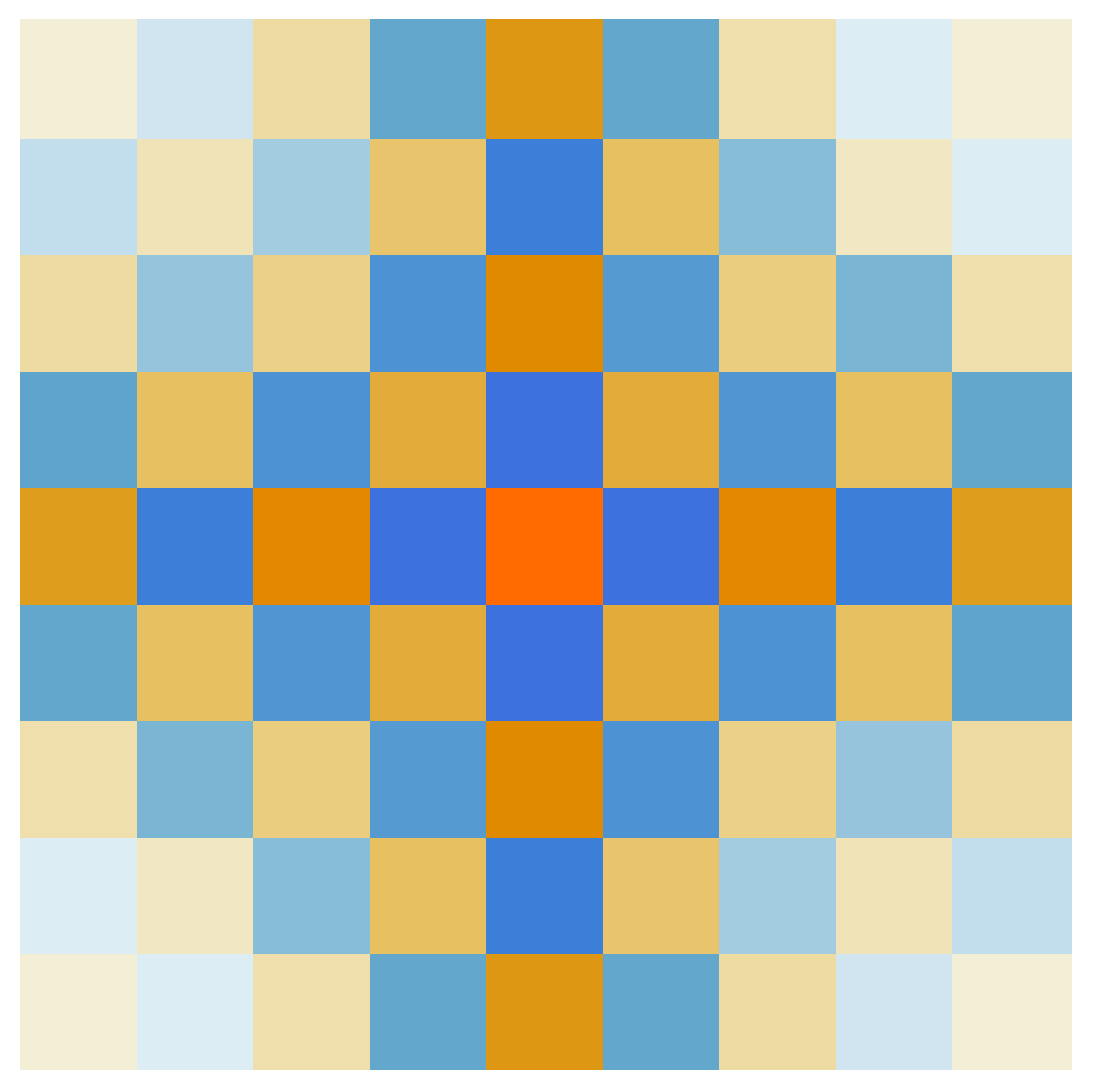}\hskip1cm\includegraphics[width=2.5in]{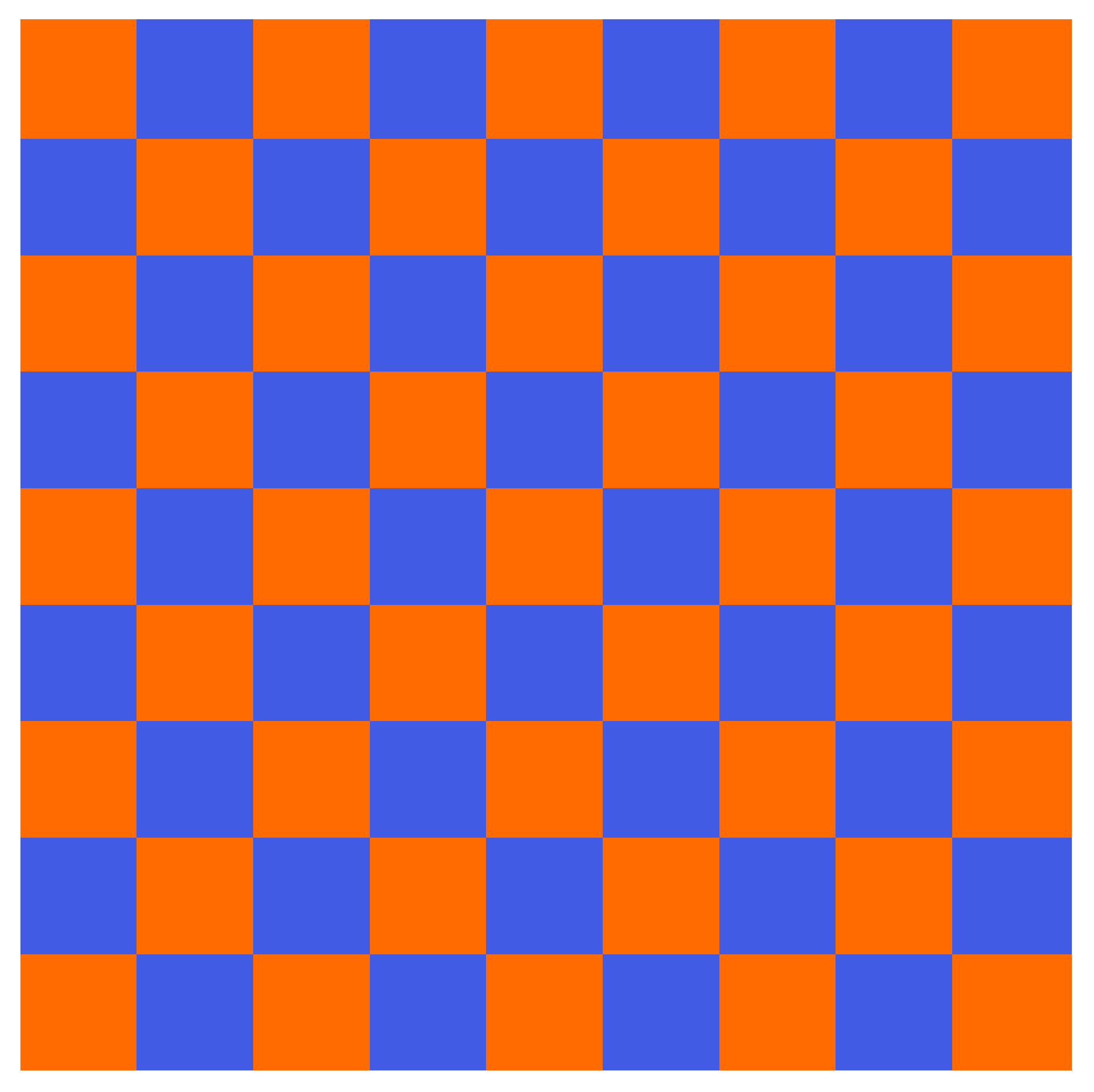}\end{center}
\caption{\small{\label{square}Covariances for the square polyomino, when $\eps=.001$ (left), and in the limit $\eps\to0$ (right).}}
\end{figure}

\subsubsection{The $+$ polyomino}
We consider the polyomino with $p(z,u) = 1+z+1/z+u+1/u.$ 
We need to compute ($Kw_0/n$ times) the Fourier series of $\frac1{\eps+|p|^2}$. 
Let $z=e^{i\theta}$ and $u=e^{i\phi}$.
The polynomial $p = 1+2\cos\theta+2\cos\phi$ vanishes on a whole curve on $\T^2$, 
where $\theta$ runs over the range $\theta\in[\pi/3,5\pi/3]$. 
For small $\eps$ the integral
concentrates near this curve. Let $(\theta_0,\phi_0)$ be a point on the curve. For $\theta$ fixed,
the denominator has the form
$\eps+a(\theta)(\phi-\phi_0)^2+O(\phi-\phi_0)^3$
where $a(\theta) = 3-4\cos\theta-4\cos^2\theta$. 

The contribution for this slice is (with $x=\phi-\phi_0$, and to leading order)
$$\frac1{2\pi}\int_{\R}\frac{dx}{\eps+a(\theta)x^2} = \frac{1}{2\sqrt{\eps a(\theta)}}.$$

We thus have to leading order (for fixed $(s,t)$ as $\eps\to0$)
$$\Cov(X_{0,0},X_{s,t}) = \frac{Kw_0}{2\pi\sqrt{\eps}}\int_{\pi/3}^{5\pi/3}\frac{\cos(s\theta+t\phi)d\theta}{\sqrt{3-4\cos\theta-4\cos^2\theta}}$$
where $\cos\phi+\cos\theta=1/2.$ 

Multiplied by $\sqrt{\eps}$ these covariances still tend to zero as $|s|+|t|\to\infty$, although the decay is slow,
of order $1/(s^2+t^2)^{1/4}$.
See Figure \ref{plus} for a numerical plot.

\begin{figure}
\begin{center}\includegraphics[width=3.in]{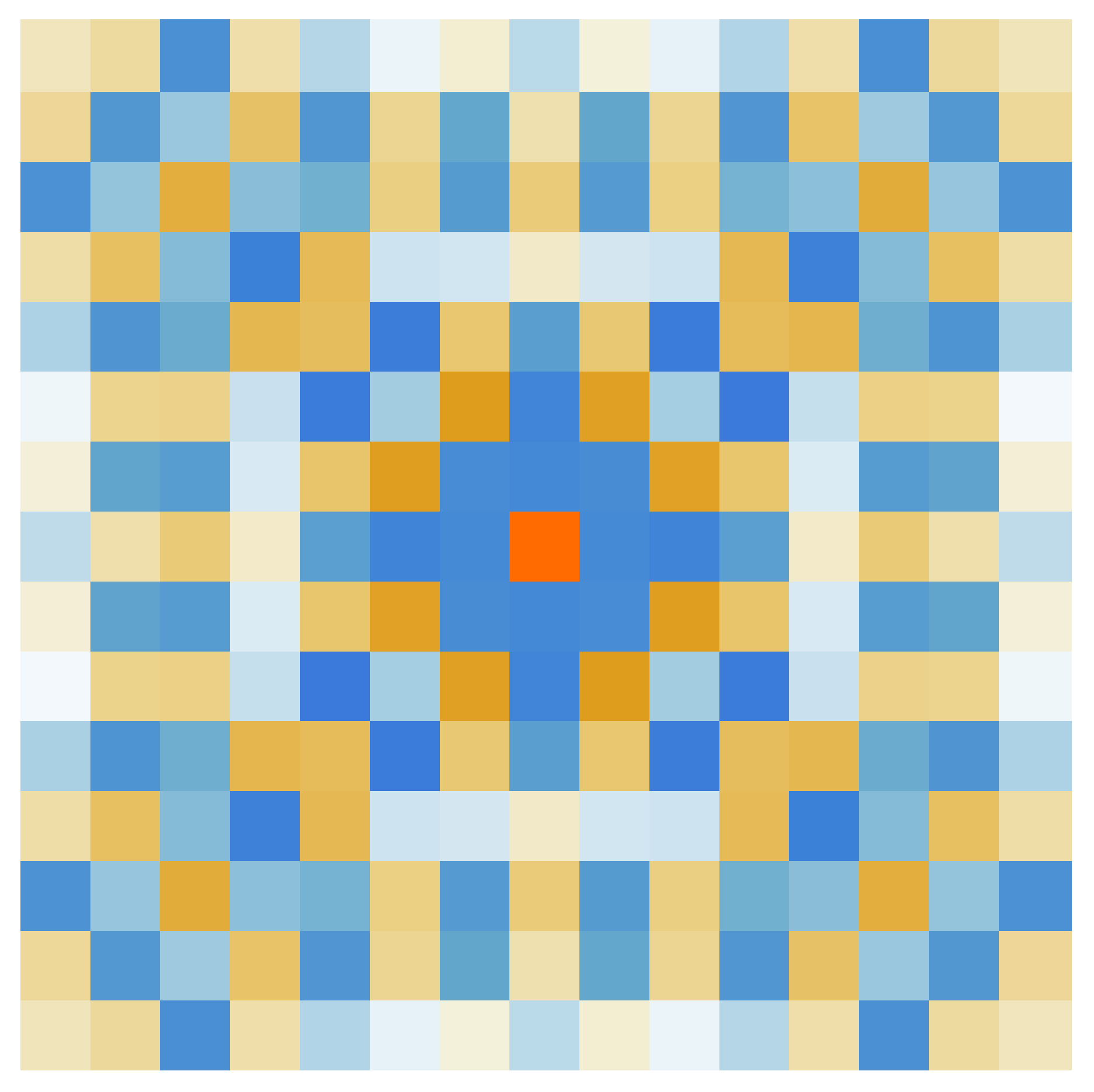}\end{center}
\caption{\small{\label{plus}Covariances for the ``plus" polyomino, in the limit $\eps\to0$.}}
\end{figure}

\section{Further directions}

Because tilings are so diverse, there are many directions for further research. Here are some ideas.

\begin{enumerate}
\item How is the covariance function for a 3D polyomino different? Typically the characteristic polynomial
intersects the unit $3$-dimensional torus $\T^3$ in a $1$-dimensional set. Is there an analogue of Theorem \ref{crystalthm}?

\item Find interesting examples with multiple tiles in $\Z^2$. 

\item For tilings of a planar domain with copies of the $L$ polyomino and a small density of monomers
(or plus polyomino and no monomers),
understand the influence of the boundary on the tiling. Can one create situations where there
is coexistence of the multiple phases?

\item What behavior do we expect for the Coulomb gas for a general tiling problem? What about the $L$ triomino?

\item Wang tiles (squares with colored edges, tiled so that adjacent tiles share the same color)
can be used to emulate any Turing machine. What is the multinomial-tiling analog of such a computation?

\item What is the natural multinomial analogue of the random partition model? What is the limit shape of a random such partition in that model (in the appropriate limit)?

\end{enumerate}

\section{Appendix: The Bessel-K function}\label{Besselscn}

The Bessel-K function, or modified Bessel function of the second kind, $B(s)$, can be defined for $s\in\R\setminus\{0\}$ 
by the integral
$$B(s) = \frac1{2\pi}\iint_{\R^2}\frac{e^{isx}\,dx\,dy}{1+x^2+y^2}.$$
For $z\in\C$, $B(|z|)$ is the Green's function for the massive laplacian, that is, it satisfies 
$$(I-\Delta)B(|z|) = \delta_0$$
where $\delta_0$ is the point measure.
 
The function $B(s)$ has a logarithmic singularity at the origin, with expansion
$$B(s) = \log\frac1s + \log 2-\gamma_E + O(s^2\log s).$$

An integral of the form 
$$\frac1{2\pi}\int_{\R^2}\frac{e^{i(sx+ty)}\,dx\,dy}{\eps+ax^2+bxy+cy^2}$$
where the quadratic form $ax^2+bxy+cy^2$ is positive definite,
can be converted into a Bessel-K integral with a linear change of coordinates,
yielding
$$=\frac1{\sqrt{c-\frac{b^2}{4a}}}B\left(\sqrt{\frac{\eps(cs^2-bst+at^2)}{ac-b^2/4}}\right).$$

\end{document}